\documentclass{compositio_mod}

\usepackage{amsgen,amsmath,amsxtra,amsfonts}
\numberwithin{equation}{section}

\usepackage{tikzexternal}
\usepackage[parfill]{parskip}
\begingroup
  \makeatletter
  \@for\theoremstyle:=definition,remark,plain\do{%
    \expandafter\g@addto@macro\csname th@\theoremstyle\endcsname{%
      \addtolength\thm@preskip\parskip
    }%
  }
\endgroup
\usepackage{nowidow}
\usepackage{microtype}
\usepackage{combelow}
\usepackage[mathscr]{euscript}

\usepackage{xcolor}
\usepackage[pdftex,colorlinks=true,allcolors=blue!70!black,breaklinks]{hyperref}
\usepackage{url}

\usepackage{enumitem}
\usepackage{booktabs}

\theoremstyle{plain}
\newtheorem{definition}{Definition}[section]
\newtheorem{proposition}[definition]{Proposition}
\newtheorem{lemma}[definition]{Lemma}
\newtheorem{theorem}[definition]{Theorem}
\newtheorem{corollary}[definition]{Corollary}
\theoremstyle{definition}
\newtheorem{remark}[definition]{Remark}

\DeclareMathOperator{\sEnd}{\mathscr{E}\mkern-1.1mu\mathit{nd}}
\DeclareMathOperator{\Map}{\mathbf{Map}}
\newcommand{\ZZ}{\mathbb{Z}}
\newcommand{\CC}{\mathbb{C}}
\newcommand{\HH}{\mathbb{H}}
\newcommand{\PP}{\mathbb{P}}
\newcommand{\GG}{\mathbb{G}}
\renewcommand{\O}{\mathcal{O}}
\newcommand{\E}{\mathcal{E}}
\newcommand{\F}{\mathcal{F}}

\begin{document}

\title{A twisted nonabelian Hodge correspondence}
\author{Alberto Garc{\'\i}a-Raboso}
\email{agraboso@math.toronto.edu}
\address{\scshape
  Department of Mathematics\\
  University of Toronto\\
  40 St. George St., room 6290\\
  Toronto, ON M5S 2E4, Canada}

\begin{abstract}
\vspace{2pt}
We prove an extension of the nonabelian Hodge theorem \cite{MR1179076} in which the underlying objects are \emph{twisted} torsors over a smooth complex projective variety. In the prototypical case of $GL_n$-torsors, one side of this correspondence consists of vector bundles equipped with an action of a sheaf of twisted differential operators in the sense of Be{\u\i}linson and Bernstein \cite{MR1237825}; on the other side, we endow them with appropriately defined twisted Higgs data.\\[.5\baselineskip]
The proof we present here is formal, in the sense that we do not delve into the analysis involved in the classical nonabelian Hodge correspondence. Instead, we use homotopy-theoretic methods ---chief among them the theory of principal $\infty$-bundles \cite{arXiv1207.0248}--- to reduce our statement to classical (untwisted) Hodge theory \cite{MR1978713}.
\end{abstract}

\maketitle
\microtypesetup{protrusion=false}
\tableofcontents
\microtypesetup{protrusion=true}
\clearpage


The study of twisted sheaves has experienced a resurgence in the last decade. Introduced by Giraud \cite{MR0344253} in his work on nonabelian cohomology in the early 1970s, they were promptly forgotten. It was not until the turn of the century that C{\u a}ld{\u a}raru \cite{MR2700538} undertook a systematic study of their derived categories and developed a theory of Fourier-Mukai transforms between them, bringing some attention back to them. A few of the more interesting recent developments concerning twisted sheaves include the following.
\begin{itemize}[itemsep=-1pt]
\item A conjecture of C{\u a}ld{\u a}raru's relating twisted Fourier-Mukai transforms to Hodge isometries of K3 lattices was proved a few years later by Huybrechts and Stellari \cite{MR2179782,MR2310257}.
\item Kontsevich \cite{MR1403918} had conjectured that they effect first-order infinitesimal deformations of categories of coherent sheaves ---an idea that was put on firm ground by Lowen and van den Bergh \cite{MR2183254,MR2238922} (see also \cite{MR2477894}). Ben-Bassat, Block and Pantev \cite{MR2309993} extended this to formal deformations in the case of complex tori, while Sawon \cite{arXiv1209.3202} recently constructed actual families of generalized K3 surfaces realizing them.
\item Donagi and Pantev \cite{MR2399730} proved a duality theorem for genus one fibrations that connects twisted sheaves to Tate-Shafarevich groups ---a kind of Pontrjagin duality for commutative group stacks, as Arinkin reflects in an appendix to \emph{loc.cit}.
\item On a different direction, Lieblich \cite{MR2309155} and Yoshioka \cite{MR2306170} constructed moduli spaces of twisted sheaves satisfying appropriate stability conditions.
\item A (unpublished\footnote{See \url{http://mathoverflow.net/questions/158614}.}) theorem of Gabber's upholding an old conjecture of Grothendieck's \cite{MR0244269} about the relationship between the Azumaya and cohomological Brauer groups of quasi-projective varieties was reproved by de Jong \cite{deJong03} using twisted sheaves.
\end{itemize}

It is only natural to explore what Hodge theory might be able to say about twisted sheaves. After all, one of the areas in which nonabelian cohomology has featured quite prominently is in the study of nonabelian Hodge theory (see, e.g., \cite{MR1397992,MR1978713}).

Abstract considerations notwithstanding, the original motivation for this work came from the tamely ramified version of the Geometric Langlands Correspondence (GLC). The program set forth in \cite{MR2537083} advocates viewing the GLC as a quantization of a certain Fourier--Mukai transform ---the quantization being mediated by appropriate nonabelian Hodge correspondences. In the compact case, it is the classical nonabelian Hodge theorem of Simpson's \cite{MR1179076} that applies. In the tamely ramified case, on the other hand, the objects that should appear are some sort of twisted bundles with parabolic structure along a divisor.

Although some approaches to defining twisted vector bundles with twisted flat connections have been made before \cite{MR1405064,ChatterjeeThesis,MR2362847,MR2045884}, the Higgs side has remained virtually unexplored until now ---as has the case of principal bundles (for which we prefer the term \emph{torsor}). In this paper we propose a definition of twisted torsors with twisted connections and twisted Higgs fields that results in a twisted nonabelian Hodge correspondence.

Our results represent a first step towards the correspondence needed to attack the GLC in the case of tame ramification: here we only deal with twisted torsors over smooth complex projective varieties, leaving the parabolic case to future work.
\noclub

Even though our main theorems ---Theorem \ref{theorem:MainTheoremGLn} for vector bundles, and Theorem \ref{theorem:MainTheorem} in the general case--- are relatively easy to understand, their full statement and our proofs use the language and the machinery of $\infty$-topoi \cite{MR2137288,MR2394633,MR2522659}, and in particular the beautiful theory of principal $\infty$-bundles \cite{arXiv1207.0248,arXiv1207.0249}. In the hopes that the reader unfamiliar with this rather abstract framework will still want to get a feeling for them, we have devoted \S\ref{section:Introduction} to presenting ``{\v C}ech''-like versions of them (Theorems \ref{theorem:MainTheoremGLnCech} and \ref{theorem:MainTheoremCech}) as candidly as possible, emphasizing the train of thought rather than a strictly logical order of exposition. We defer a discussion of the remaining contents of this paper until the end of that section.

\begin{acknowledgements}
The results of this paper constitute my doctoral dissertation at the University of Pennsylvania. I want to thank my advisor, Tony Pantev, for suggesting the problem, for his guidance in my investigation of it, for his constant support and encouragement, and for coming up with counterexamples for some of my most outlandish and stupid claims ---just one of the faces of his immense mathematical prowess; the University of Pennsylvania, for its generous financial support in the form of a Benjamin Franklin fellowship; Carlos T. Simpson, for his continued interest in my work; Urs Schreiber, for patiently listening to me, and for creating and maintaining a resource as useful as the nLab; Marc Hoyois, for ever mentioning the words \emph{1-localic $\infty$-topoi}, and for answering some of my questions; Angelo Vistoli, for his help with Lemma \ref{lemma:EtaleQuotientAlgebraicGroups}; Nikita Nikolaev for a thorough proofreading and many useful comments on language; Tyler L. Kelly, Drago{\cb s} Deliu, Umut Isik, Pranav Pandit and Ana Pe{\' on}-Nieto, for many helpful conversations; and my wife, Roc{\' i}o Naveiras-Cabello, for more reasons than I could list here.

I acknowledge support from NSF grants DMS 1107452, 1107263, 1107367 ``RNMS: GEometric structures And Representation varieties'' (the GEAR Network), the NSF RTG grant DMS 0636606 and the NSF grant DMS 1001693.
\end{acknowledgements}

\section{Introduction}
\label{section:Introduction}

\subsection{Twisted vector bundles}
\label{section:TwistedVectorBundles}

\subsubsection{}

Let $X$ be a smooth projective variety over $\CC$, considered either as a scheme with the {\' e}tale topology or as a complex analytic space endowed with the classical topology. Given $\alpha \in H^2(X, \O^\times_X)$, we can always\footnote{\label{footnote:GoodCovers}In the analytic case, take a good open cover; for the {\' e}tale topology, see \cite[Theorem III.2.17]{MR559531}.} choose an open cover $\mathfrak{U} = \{ U_i \to X \}_{i \in I}$ of $X$ such that there exists a {\v C}ech 2-cocycle $\underline{\alpha} = \{ \alpha_{ijk} \} \in \check{Z}^2(\mathfrak{U}, \O^\times_X)$ representing the class $\alpha$. The following definition goes back to Giraud's work on nonabelian cohomology \cite{MR0344253}.

\begin{definition}
\label{definition:TwistedSheaf}
An \emph{$\alpha$-twisted sheaf} on $X$ is a collection
\[
\Big( \underline{\E} = \{ \E_i \}_{i \in I}, \; \underline{g} = \{ g_{ij} \}_{i, j \in I} \Big)
\]
of sheaves $\E_i$ of $\O_X$-modules on $U_i$, together with isomorphisms $g_{ij}: \E_j |_{U_{ij}} \to \E_i |_{U_{ij}}$ satisfying $g_{ii} = \operatorname{id}_{\E_i}$, $g_{ij} = g_{ji}^{-1}$, and the $\underline{\alpha}$-twisted cocycle condition,
\[
g_{ij} g_{jk} g_{ki} = \alpha_{ijk} \operatorname{id}_{\E_i},
\]
on $U_{ijk}$ for any $i, j, k \in I$. Given two $\alpha$-twisted sheaves, $\big( \underline{\E}, \underline{g} \big)$ and $\big( \underline{\F}, \underline{h} \big)$, a morphism between them is given by a collection $\underline{\varphi} = \{ \varphi_i\}_{i \in I}$ of morphisms $\varphi_i : \E_i \to \F_i$ intertwining the transition functions; i.e., such that $\varphi_i g_{ij} = h_{ij} \varphi_j$.
\end{definition}

Although this definition uses a particular representative for the class $\alpha$, it can be shown \cite{MR2700538} that the category of such objects is independent ---up to equivalence--- of the choice of cover $\mathfrak{U}$ and representing cocycle $\underline{\alpha}$, justifying the choice of labeling them with the class $\alpha$ in sheaf cohomology instead of the cocycle $\underline{\alpha}$ itself.

In what follows we will be concerned not with the whole category of $\alpha$-twisted sheaves but with the subgroupoid $_{\alpha}\mathfrak{Vec}_n(X)$ of $\alpha$-twisted vector bundles of fixed rank $n$ and isomorphisms between them; to wit, those $\alpha$-twisted sheaves $\big( \underline{\E}, \underline{g} \big)$ for which each $\E_i$ is a vector bundle of rank $n$, with morphisms the invertible ones.

\subsubsection{}

Definition \ref{definition:TwistedSheaf} makes it clear that the twisting $\underline{\alpha}$ should be considered as the substrate on which $\alpha$-twisted vector bundles live.

\begin{definition}
\label{definition:GmGerbeCech}
A $\mathfrak{U}$-$\GG_m$-gerbe over $X$ is the datum of a 2-cocycle $\underline{\alpha} \in \check{Z}^2(\mathfrak{U}, \GG_m)$.
\end{definition}

$\mathfrak{U}$-$\GG_m$-gerbes form a strict 2-category ---a 2-groupoid, in fact---, with the 1-morphisms given by {\v C}ech 1-cochains and the 2-morphisms by {\v C}ech 0-cochains. This category is not independent of the choice of cover $\mathfrak{U}$; nevertheless, taking a refinement $\mathfrak{V}$ of $\mathfrak{U}$ produces a functor from the category of $\mathfrak{U}$-$\GG_m$-gerbes to that of $\mathfrak{V}$-$\GG_m$-gerbes\footnote{This functor is injective at the level of equivalence classes of objects if we restrict to an appropriate class of covers (see footnote \ref{footnote:GoodCovers}).}.

We will later give a more abstract definition ---that of a \emph{$\GG_m$-gerbe} (Definitions \ref{definition:bandedGerbes} and \ref{definition:AGerbesAsPrincipalInftyBundles})--- that does not depend on the choice of a cover, and which provides a true geometric realization for classes in sheaf cohomology.

\begin{definition}
\label{definition:BasicVectorBundleOverGmGerbeCech}
A \emph{basic\footnote{Classically these are known as \emph{weight 1 vector bundles} on $\underline{\alpha}$. In the more general setting in which we wish to prove our result ---twisted torsors for groups other than $GL_n$---, a simple weight does not suffice; hence the need for a different terminology. See \S\ref{section:TorsorsOverGerbes}.} vector bundle on a $\mathfrak{U}$-$\GG_m$-gerbe $\underline{\alpha}$ over $X$} is an $\alpha$-twisted vector bundle $(\underline{\E}, \underline{g})$ on $X$.
\end{definition}

At this point, this definition seems unnecessary ---and we will, in fact, keep using the term $\alpha$-twisted vector bundle. Its convenience will become manifest in just a few short pages when we introduce twisted connections and twisted Higgs fields.

\subsubsection{}
\label{section:RemarksAboutTorsionAndProjectivization}

We make two important remarks about this category of $\alpha$-twisted vector bundles.

\begin{itemize}
\item For a fixed class $\alpha \in H^2(X, \O^\times_X)$, there may be no $\alpha$-twisted vector bundles of rank $n$. In general this is a difficult question concerning the relationship between the Azumaya Brauer group of $X$ and its cohomological Brauer group \cite{MR0244269}. A short survey of these issues can be found in \cite[\S2.1.3]{MR2399730}. We simply observe that $_{\alpha}\mathfrak{Vec}_n(X)$ is empty unless $\alpha$ is $n$-torsion. Indeed, if there is an object $\big( \underline{\E}, \underline{g} \big) \in\!\!\ _{\alpha}\mathfrak{Vec}_n(X)$, then $\det\underline{g} = \{ \det g_{ij} \}_{i, j \in I}$ provides an element of $\check{C}^1(\mathfrak{U}, \O^\times_X)$ whose \v Cech differential equals $\underline{\alpha}^n$.
\item From an $\alpha$-twisted vector bundle of rank $n$ we can produce an honest, untwisted $\PP^{n-1}$-bundle by projectivizing all the locally defined vector bundles: the twisting $\underline{\alpha}$ goes away because it is contained in the kernel of the map $GL_n \to \PP GL_n$ ---which coincides with the center of $GL_n$. This construction is clearly functorial.
\end{itemize}
These facts will appear once again when we talk about twisted connections and twisted Higgs fields below. Their recurrence will be explicated thoroughly in \S \ref{section:TorsionPhenomena}.

\subsection{Twisted connections and twisted Higgs fields}
\label{section:TwistedConnectionsAndTwistedHiggsFields}

\subsubsection{The (classical) nonabelian Hodge theorem.}

Let $\E$ be a vector bundle on $X$. Recall that a \emph{connection} on $\E$ is a $\CC$-linear map $\nabla : \E \to \Omega^1_X \otimes \E$ satisfying the Leibniz rule
\[
\nabla (f v) = df \otimes v + f \nabla v,  \qquad\text{for } f \in \O_X, v \in \E.
\]
A connection $\nabla$ naturally extends to a collection of $\CC$-linear maps $\nabla^{(k)} : \Omega^k_X \otimes \E \to \Omega^{k+1}_X \otimes \E$ defined by the graded Leibniz identity
\[
\nabla^{(k)}(\alpha \otimes v) = d\alpha \otimes v + (-1)^k \alpha \wedge \nabla v, \qquad\text{for } v \in \E, \alpha \in \Omega^k_X.
\]
The compositions $\nabla^{(k+1)} \circ \nabla^{(k)}$ are then $\O_X$-linear, and we define the \emph{curvature $C(\nabla)$} of $\nabla$ to be the image of $\nabla^{(1)} \circ \nabla$ under the standard duality isomorphism
\begin{equation} \label{eq:HomDuality}
\operatorname{Hom}_{\O_X}(\E, \Omega^2_X \otimes \E) \cong \operatorname{Hom}_{\O_X}\!\big( \O_X, \Omega^2_X \otimes \sEnd(\E) \big) \cong \Gamma \big( X, \Omega^2_X \otimes \sEnd(\E) \big).
\end{equation}
A connection is said to be \emph{flat} if its curvature vanishes; it is then customary to call the pair $\big( \E, \nabla \big)$ a \emph{flat vector bundle}. The difference of any two connections $\nabla_1$ and $\nabla_2$ on the same vector bundle $\E$ is an $\O_X$-linear map and can also be considered as a 1-form with values in the endomorphism bundle of $\E$ through the isomorphisms
\[
\operatorname{Hom}_{\O_X}(\E, \Omega^1_X \otimes \E) \cong \operatorname{Hom}_{\O_X}\!\big( \O_X, \Omega^1_X \otimes \sEnd(\E) \big) \cong \Gamma \big( X, \Omega^1_X \otimes \sEnd(\E) \big).
\]

On the other hand, a \emph{Higgs bundle} is a pair $\big( \F, \phi \big)$ of a vector bundle $\F$ together with an $\O_X$-linear map $\phi : \F \to \Omega^1_X \otimes \F$ ---usually referred to as a \emph{Higgs field}--- satisfying
\[
0 = \phi \wedge \phi \in \Gamma \big( X, \Omega^2_X \otimes \sEnd(\F) \big),
\]
where $\phi \wedge \phi$ is called the \emph{curvature $C(\phi)$} of $\phi$ in analogy with the case of a connection. Two Higgs fields $\phi_1$ and $\phi_2$ on the same vector bundle $\F$ also differ by a 1-form with values in $\sEnd(\F)$.

The nonabelian Hodge theorem \cite{MR1179076} (see also \cite{arXiv1406.1693} for a review) establishes an equivalence between the category of flat vector bundles on $X$ and a certain full subcategory of the category of Higgs bundles on the same variety. The latter is specified by two conditions on objects:
\begin{itemize}
\item The first one is purely topological: the components of the first and second Chern characters of $\mathcal{F}$ along the hyperplane class $[H] \in H^2(X, \CC)$ of $X$ (which we refer to as the Chern \emph{numbers} of $\mathcal{F}$) should vanish:
\[
\mathrm{ch}_1(\mathcal{F}) \cdot [H]^{\mathrm{dim}\,X-1} = 0 = \mathrm{ch}_2(\mathcal{F}) \cdot [H]^{\mathrm{dim}\,X-2}
\]
Equivalently, the first and second Chern \emph{classes} of $\mathcal{F}$ vanish along $[H]$.
\item The second condition depends on the holomorphic structure of the underlying vector bundle as well as on the Higgs field: a Higgs bundle $\big( \F, \phi \big)$ is said to be \emph{semistable} if for every subbundle $\F' \subset \F$ preserved by the Higgs field ---i.e, such that $\phi(\F') \subseteq \Omega^1_X \otimes \F'$--- we have $\mu(\F') \leq \mu(\F)$. Here the \emph{slope} $\mu$ of a vector bundle is defined as the quotient of its degree by its rank.
\end{itemize}
The first of these conditions implies the vanishing of the slope of any Higgs bundle in this subcategory, since $\mathrm{ch}_1(\mathcal{F}) \cdot [H]^{\mathrm{dim}\,X-1}$ is its degree; the second condition then reduces to saying that any $\phi$-invariant subbundle of $\mathcal{F}$ has non-positive degree.

\subsubsection{Twisted connections.}

The first step towards formulating a twisted version of the nonabelian Hodge theorem is to determine what a ``flat connection'' on an an $\alpha$-twisted vector bundle $\big( \underline{\E}, \underline{g} \big)$ should be. The na{\" i}ve definition consists of equipping each $\E_i$ with a flat connection $\nabla_i : \E_i \to \Omega^1_{U_i} \otimes \E_i$ in such a way that the following diagram is commutative for every $i, j \in I$:
\[
\begin{tikzpicture}[x=100pt, y=40pt]
  \node (11) at (0,1) {$\E_j \big|_{U_{ij}}$};
  \node (12) at (1,1) {$\Omega^1_{U_j} \otimes \E_j \big|_{U_{ij}}$};
  \node (21) at (0,0) {$\E_i \big|_{U_{ij}}$};
  \node (22) at (1,0) {$\Omega^1_{U_i} \otimes \E_i \big|_{U_{ij}}$};
  \path[semithick,->]
    (11) edge node[scale=0.8,yshift=10pt] {$\nabla_j$} (12)
    ([xshift=-8pt]11.south) edge node[pos=0.45,scale=0.8,xshift=10pt] {$g_{ij}$} ([xshift=-8pt]21.north)
    ([xshift=6pt]12.south) edge node[pos=0.45,scale=0.8,xshift=10pt] {$g_{ij}$} ([xshift=6pt]22.north)
    (21) edge node[scale=0.8,yshift=10pt] {$\nabla_i$} (22);
\end{tikzpicture}
\]
More compactly, $\nabla_i - g_{ij} \nabla_j \, g_{ij}^{-1} = 0$.

There is, however, a natural weakening of these requirements that yields a more general and interesting class of objects: namely, to allow for the locally defined connections to
\begin{itemize}[topsep=-3pt,partopsep=0pt,parsep=0pt,itemsep=1pt]
\item differ on double intersections by 1-forms with values in the center of the appropriate endomorphism bundle, and
\item have nonzero central curvature.
\end{itemize}
This amounts to choosing cochains
\[
\underline{\omega} = \{ \omega_{ij} \} \in \check{C}^1(\mathfrak{U}, \Omega^1_X) \quad\text{and}\quad \underline{F} = \{ F_i \} \in \check{C}^0(\mathfrak{U}, \Omega^2_X)
\]
diagonally embedded in $\check{C}^1 \big( \mathfrak{U}, \Omega^1_X \otimes \sEnd(\E) \big)$ and $\check{C}^0 \big( \mathfrak{U}, \Omega^2_X \otimes \sEnd(\E) \big)$, respectively, and demanding the $\nabla_i$ to satisfy the equations
\[
\nabla_i - g_{ij} \nabla_j \, g_{ij}^{-1} = \omega_{ij} \quad\text{and}\quad C(\nabla_i) = F_i.
\]
Of course we cannot pick $\underline{\omega}$ and $\underline{F}$ arbitrarily. Rather, there is a set of compatibility conditions coming from the fact that the $\nabla_i$ are connections:
\begin{equation}  \label{eq:FlatUGmGerbe}
\begin{matrix}
\omega_{ik} = \omega_{ij} + \omega_{jk} - d\log\alpha_{ijk} & \text{on } U_{ijk} \\[1ex]
F_i - F_j =  d\omega_{ij} & \text{on } U_{ij\;\,} \\[1ex]
dF_i = 0 & \text{on } U_{i\;\;\,\,}
\end{matrix}
\end{equation}
These are precisely the relations needed to make the triple $\big( \underline{\alpha}, \underline{\omega}, \underline{F} \big)$ into a {\v C}ech 2-cocycle in hypercohomology of the multiplicative de Rham complex of $X$:
\begin{equation}
\label{eq:deRhamComplexG_m}
\mathrm{dR}^{\GG_m}_X := \Big[ \O^\times_X \xrightarrow{\;\;d\log\;\;} \Omega^1_X \xrightarrow{\quad d \quad} \Omega^2_X \xrightarrow{\quad d \quad} \cdots \Big].
\end{equation}
To fix notation, recall that, given a complex $(E^\bullet, d)$ of sheaves on $X$, its hypercohomology with respect to a cover $\mathfrak{U}$ is defined as the cohomology of the total complex of the double complex
\[
\begin{tikzpicture}[x=250pt, y=80pt]
  \node (11) at (0,1) {$\vdots$};
  \node (12) at (0.333,1) {$\vdots$};
  \node (13) at (0.666,1) {$\vdots$};
  \node (21) at (0,0.5) {$\check{C}^1(\mathfrak{U}, E^0)$};
  \node (22) at (0.333,0.5) {$\check{C}^1(\mathfrak{U}, E^1)$};
  \node (23) at (0.666,0.5) {$\check{C}^1(\mathfrak{U}, E^2)$};
  \node (24) at (0.94,0.5) {$\cdots$};
  \node (31) at (0,0) {$\check{C}^0(\mathfrak{U}, E^0)$};
  \node (32) at (0.333,0) {$\check{C}^0(\mathfrak{U}, E^1)$};  
  \node (33) at (0.666,0) {$ \check{C}^0(\mathfrak{U}, E^2)$};
  \node (34) at (0.94,0) {$\cdots$};
  \path[semithick,->]
    (21) edge node[pos=0.45,scale=0.8,xshift=8pt] {$\delta$} (11)
    (22) edge node[pos=0.45,scale=0.8,xshift=12pt] {$-\delta$} (12)
    (23) edge node[pos=0.45,scale=0.8,xshift=8pt] {$\delta$} (13)
    (21) edge node[pos=0.5,scale=0.8,yshift=8pt] {$d$} (22)
    (22) edge node[pos=0.5,scale=0.8,yshift=8pt] {$d$} (23)
    (23) edge node[pos=0.5,scale=0.8,yshift=8pt] {$d$} (24)
    (31) edge node[pos=0.45,scale=0.8,xshift=8pt] {$\delta$} (21)
    (32) edge node[pos=0.45,scale=0.8,xshift=12pt] {$-\delta$} (22)
    (33) edge node[pos=0.45,scale=0.8,xshift=8pt] {$\delta$} (23)
    (31) edge node[pos=0.5,scale=0.8,yshift=8pt] {$d$} (32)
    (32) edge node[pos=0.5,scale=0.8,yshift=8pt] {$d$} (33)
    (33) edge node[pos=0.5,scale=0.8,yshift=8pt] {$d$} (34);
\end{tikzpicture}
\]

\begin{definition}
\label{definition:FlatGmGerbeCech}
A \emph{flat $\mathfrak{U}$-$\GG_m$-gerbe over $X$} is a 2-cocycle $\big( \underline{\alpha}, \underline{\omega}, \underline{F} \big) \in \check{\ZZ}^2 \big( \mathfrak{U}, \mathrm{dR}^{\GG_m}_X \big)$.
\end{definition}

The category of flat $\mathfrak{U}$-$\GG_m$-gerbes is again a 2-groupoid, with 1- and 2-morphisms given by {\v C}ech 1- and 0-cochains, respectively. The corresponding cover-independent concept is that of a \emph{flat $\GG_m$-gerbe} (Definition \ref{definition:FlatGerbe}).

The notions of flat $\mathfrak{U}$-$\GG_m$-gerbes and flat $\GG_m$-gerbes have appeared in similar form in the literature before under the names of Dixmier-Douady sheaves of groupoids with connective structure and curving \cite{MR2362847}, bundle gerbes with connection and curvature \cite{MR1405064}, and gerbs with 0- and 1-connection \cite{ChatterjeeThesis}. Our approach in this introduction is similar to that of the last reference, while the point of view in the remainder is closer in spirit to the first two.

\begin{definition}
\label{definition:BasicVectorBundleOverFlatGmGerbeCech}
A \emph{basic vector bundle on a flat $\mathfrak{U}$-$\GG_m$-gerbe $\big( \underline{\alpha}, \underline{\omega}, \underline{F} \big)$ over $X$} is a collection $\big( \underline{\E}, \underline{\nabla}, \underline{g} \big)$ consisting of an $\alpha$-twisted vector bundle $\big( \underline{\E}, \underline{g} \big)$ on $X$ together with connections $\nabla_i$ on $\E_i$ satisfying the equations
\[
\nabla_i - g_{ij} \nabla_j \, g_{ij}^{-1} = \omega_{ij} \quad\text{and}\quad C(\nabla_i) = F_i.
\]
\end{definition}

Having fixed a flat $\mathfrak{U}$-$\GG_m$-gerbe $\big( \underline{\alpha}, \underline{\omega}, \underline{F} \big)$, we can consider the groupoid of basic vector bundles on it. Its morphisms are given by those isomorphisms of the underlying $\alpha$-twisted vector bundles that commute with the connections. A straightforward computation shows that an element of $\check{\CC}^1 \big( \mathfrak{U}, \mathrm{dR}^{\GG_m}_X \big)$ giving an equivalence between two flat $\mathfrak{U}$-$\GG_m$-gerbes also provides an equivalence of their respective categories of basic vector bundles. Arguing as in \cite{MR2700538}, it is possible to prove that refining the cover $\mathfrak{U}$ does not affect this category up to equivalence.

If the cocycle $\underline{\alpha}$ is composed of locally constant functions ---so that $d\log\underline{\alpha} = 0$---, the equations \eqref{eq:FlatUGmGerbe} satisfied by the 1- and 2-form parts of a flat $\mathfrak{U}$-$\GG_m$-gerbe do not involve $\underline{\alpha}$. In this case, we can separate the data of a flat $\mathfrak{U}$-$\GG_m$-gerbe into the bare $\mathfrak{U}$-$\GG_m$-gerbe $\underline{\alpha}$ and the pair $\big( \underline{\omega}, \underline{F} \big)$. The latter, which defines a class in $\check{\HH}^1(\mathfrak{U}, \Omega^1_X \to \Omega^{2, \mathrm{cl}}_X)$, gives rise to a sheaf of twisted differential operators (TDOs) on $X$ in the sense of Be{\u\i}linson and Bernstein \cite{MR1237825}, and basic vector bundles on $\big( \underline{\alpha}, \underline{\omega}, \underline{F} \big)$ can then be described as $\alpha$-twisted vector bundles equipped with an action of this sheaf of TDOs.

The fact that $\alpha$ is $n$-torsion (remember \S\ref{section:RemarksAboutTorsionAndProjectivization}) implies that we can always choose a representing cocycle $\underline{\alpha}$ that is indeed locally constant. Not only that, but, given a flat $\mathfrak{U}$-$\GG_m$-gerbe, we can always find an equivalent one for which the part in $\check{Z}^2(\mathfrak{U}, \O^\times_X)$ is locally constant. Hence, we can realize basic vector bundles on a flat $\mathfrak{U}$-$\GG_m$-gerbe as twisted vector bundles with an action of a sheaf of TDOs. In fact, more is true: the class of $\big( \underline{\omega}, \underline{F} \big)$ in hypercohomology (equivalently, the sheaf of TDOs it yields) must also be $n$-torsion.

There is yet one more thing that carries over from the case of bare twisted vector bundles of \S\ref{section:TwistedVectorBundles}: projectivizing kills all central data, and so a basic vector bundle of rank $n$ on a flat $\mathfrak{U}$-$\GG_m$-gerbe yields a $\PP^{n-1}$-bundle with flat connection on $X$.

\subsubsection{Twisted Higgs fields.}

After the discussion above, it is clear how we should proceed on the Higgs bundle side. Given an $\alpha'$-twisted vector bundle, $\big( \underline{\E}', \underline{g'} \big)$, pick cochains
\[
\underline{\omega}' = \{ \omega'_{ij} \} \in \check{C}^1(\mathfrak{U}, \Omega^1_X) \quad\text{and}\quad \underline{F}' = \{ F'_i \} \in \check{C}^0(\mathfrak{U}, \Omega^2_X),
\]
equip each $\E'_i$ with a Higgs field, $\phi_i$, and require that they fulfill the equations
\[
\phi_i - g'_{ij} \phi_j \, (g'_{ij})^{-1} = \omega'_{ij} \quad\text{and}\quad C(\phi_i) = F'_i.
\]
The compatibility conditions in this case are
\[
\begin{matrix}
\omega'_{ik} = \omega'_{ij} + \omega'_{jk} & \text{on } U_{ijk} \\[1ex]
F'_i = F'_j & \text{on } U_{ij\;\,}
\end{matrix}
\]
which say that the triple $\big( \underline{\alpha}', \underline{\omega}', \underline{F}' \big)$ assembles into a {\v C}ech 2-cocycle in hypercohomology of the multiplicative Dolbeault complex:
\[
\mathrm{Dol}^{\GG_m}_X := \Big[ \O^\times_X \xrightarrow{\quad 0 \quad} \Omega^1_X \xrightarrow{\quad 0 \quad} \Omega^2_X \xrightarrow{\quad 0 \quad} \cdots \Big].
\]

The following parallel Definitions \ref{definition:FlatGmGerbeCech} and \ref{definition:BasicVectorBundleOverFlatGmGerbeCech}, and the comments below those about the corresponding categories apply verbatim, as well as our recurring remarks about torsion and projectivization (cf. \S\ref{section:RemarksAboutTorsionAndProjectivization})

\begin{definition}
\label{definition:HiggsGmGerbeCech}
A \emph{Higgs $\mathfrak{U}$-$\GG_m$-gerbe} over $X$ is a 2-cocycle $\big( \underline{\alpha}', \underline{\omega}', \underline{F}' \big) \in \check{\ZZ}^2 \big( \mathfrak{U}, \mathrm{Dol}^{\GG_m}_X \big)$.
\end{definition}

\begin{definition}
\label{definition:BasicVectorBundleOverHiggsGmGerbeCech}
A \emph{basic vector bundle on a Higgs $\mathfrak{U}$-$\GG_m$-gerbe $\big( \underline{\alpha}', \underline{\omega}', \underline{F}' \big)$ over $X$} is a collection $\big( \underline{\E}', \underline{\phi}, \underline{g}' \big)$ consisting of an $\alpha'$-twisted vector bundle $\big( \underline{\E}', \underline{g}' \big)$ on $X$ together with Higgs fields $\phi_i$ on $\E'_i$ satisfying the equations
\[
\phi_i - g'_{ij} \phi_j \, (g'_{ij})^{-1} = \omega'_{ij} \quad\text{and}\quad C(\phi_i) = F'_i.
\]
\end{definition}

\subsubsection{The nonabelian Hodge correspondence for twisted vector bundles.}

We are finally in a position to state a (weak) form of the first main theorem of this paper.

\begin{theorem}
\label{theorem:MainTheoremGLnCech}
Let $\big( \underline{\alpha}, \underline{\omega}, \underline{F} \big)$ be a flat $\mathfrak{U}$-$\GG_m$-gerbe over $X$. Then there is a Higgs $\mathfrak{U}$-$\GG_m$-gerbe, $\big( \underline{\alpha}', \underline{\omega}', \underline{F}' \big)$, over $X$ for which there is a fully faithful functor
\[
\begin{Bmatrix}
\text{Basic vector bundles of rank } n \\ \text{on } \big( \underline{\alpha}, \underline{\omega}, \underline{F} \big)
\end{Bmatrix}
\lhook\joinrel\longrightarrow
\begin{Bmatrix}
\text{Basic vector bundles of rank } n \\ \text{on } \big( \underline{\alpha}', \underline{\omega}', \underline{F}' \big)
\end{Bmatrix}
\]
Conversely, given a Higgs $\mathfrak{U}$-$\GG_m$-gerbe, $\big( \underline{\alpha}', \underline{\omega}', \underline{F}' \big)$, there exists a flat $\mathfrak{U}$-$\GG_m$-gerbe, $\big( \underline{\alpha}, \underline{\omega}, \underline{F} \big)$, for which the same conclusion holds.
\end{theorem}

\subsection{Groups other then \texorpdfstring{$GL_n$}{GL\_ n}}
\label{section:GroupsOtherThanGLn}

\subsubsection{The (classical) nonabelian Hodge theorem for torsors.}

Let $G$ be a linear algebraic group over $\CC$, viewed either as a group scheme in the {\'e}tale topology or as complex Lie group, and denote $\mathfrak{g} = \operatorname{Lie} G$.

Let $P \to X$ be an $G$-torsor\footnote{Going forward we will often use the same letter for denoting a linear algebraic group over $\CC$ and the sheaf on $X$ obtained by pullback. The reader should use the second interpretation in order to bring our use of the term torsor in line with the usual definition in algebraic geometry. As an example of future appearances of this convention, we write $\GG_m$-gerbes for what some authors (e.g., \cite{MR2399730}) call $\O^\times_X$-gerbes.} on $X$. Denote by $R: P \times G \to P$ the canonical right action of $G$ on $P$, and by $R_g: P \to P$ ($g \in G$) and $R_p: G \to P$ ($p \in P$) the obvious restrictions. A \emph{connection} on $P \to X$ is a global section $\eta \in H^0(P, \Omega^1_P \otimes \mathfrak{g})$ satisfying the following two conditions:
\begin{itemize}[topsep=-3pt,partopsep=0pt,parsep=0pt,itemsep=1pt]
\item (\emph{$Ad$-equivariance}) under right multiplication by $G$, the connection form $\eta$ transforms via the adjoint representation of $G$ on $\mathfrak{g}$, i.e., $(R_g)^\ast \eta = Ad_{g^{-1}} (\eta)$ for every $g \in G$; and
\item for every $p \in P$, the pullback $(R_p)^\ast \eta$ coincides with the Maurer--Cartan form of $G$.
\end{itemize}
The difference between two connections, $\eta_1$ and $\eta_2$, is, evidently, a global section of the same bundle, while the \emph{curvature} of a connection $\eta$ is defined through the Cartan formula:
\[
C(\eta) = d\eta + \frac{1}{2}\, [\eta, \eta] \in H^0(P, \Omega^2_P \otimes \mathfrak{g}),
\]
where $[-,-]$ is the symmetric bilinear product consisting of the Lie bracket on $\mathfrak{g}$ and the wedge product of one-forms. A $G$-torsor equipped with a connection of vanishing curvature is called \emph{flat}.

On the surface, this seems different than the case of vector bundles: here the relevant forms live on the total space of the torsor rather than on the base. However, they are of a very special kind: both $\eta_1 - \eta_2$ and $C(\eta)$ are $Ad$-equivariant and horizontal ---that is, their pullbacks by $R_p$ vanish for every $p \in P$. Forms on $P$ satisfying these two conditions are called \emph{tensorial}\footnote{The term \emph{basic} is also in use in the literature, but it conflicts with our use of that word.}, and are in one-to-one correspondence with forms on $X$ with values in the adjoint bundle $Ad\, P = P \times_{Ad} \mathfrak{g} \to X$. We shall henceforth make no notational distinction between forms on $X$ with values in $Ad\, P$ and their corresponding forms on $P$ with values in $\mathfrak{g}$.

A \emph{Higgs $G$-torsor} is a $G$-torsor $Q \to X$ equipped with a global section $\phi \in H^0(X, \Omega^1_X \otimes Ad\, Q)$. The curvature $C(\phi)$ of $\phi$ is defined as $[\phi, \phi] \in H^0(X, \Omega^2_X \otimes Ad\, Q)$. Differences between Higgs fields also belong to $H^0(X, \Omega^1 \otimes Ad\, Q)$.

A version of the nonabelian Hodge theorem for $G$-torsors follows formally from that of vector bundles and Tannakian considerations \cite[\S 6]{MR1179076}. Roughly, a linear representation $\rho: G \to \operatorname{End} V$ determines a pair of functors mapping flat $G$-torsors  (resp., Higgs $G$-torsors) to flat bundles (resp., Higgs bundles) ---at the level of the underlying, bare torsors, this is just the associated bundle construction. A Higgs $G$-torsor is said to be semistable and to have vanishing Chern numbers if so does the Higgs bundle associated to it by any representation of $G$. We then have an equivalence between the category of flat $G$-torsors on $X$, on the one hand, and that of semistable Higgs $G$-torsors with vanishing Chern numbers on the other.

\subsubsection{The nonabelian Hodge correspondence for twisted torsors.}

Let $H$ be a linear algebraic group over $\CC$, and $A \subset H$ a closed central subgroup\footnote{Nothing is lost by assuming that $A$ is, in fact, the whole center of $H$.}. As in \S\ref{section:TwistedVectorBundles}, we view them either as group schemes in the {\'e}tale topology, or as complex Lie groups endowed with their analytic topology. In both cases, the quotient $K = H/A$ is again a linear algebraic group over $\CC$ (see \S\ref{section:LinearAlgebraicGroups}). Denote by $\mathfrak{h}$, $\mathfrak{a}$ and $\mathfrak{k}$ the Lie algebras of $H$, $A$ and $K$, respectively.

Since $A$ is central in $H$, so is $\mathfrak{a}$ in $\mathfrak{h}$. The adjoint bundle of any $H$-torsor thus contains the trivial vector bundle with fiber $\mathfrak{a}$ as a subbundle. In particular, we can consider forms on the base with values in $\mathfrak{a}$ as $Ad$-equivariant, horizontal forms on the total space of the $H$-torsor.

We now state the obvious generalizations of Definitions \ref{definition:GmGerbeCech}--\ref{definition:BasicVectorBundleOverHiggsGmGerbeCech} without further comment.

\begin{definition}
[(cf. Definition \ref{definition:GmGerbeCech})]
\label{definition:AGerbeCech}
A $\mathfrak{U}$-$A$-gerbe over $X$ is a 2-cocycle $\underline{\alpha} \in \check{Z}^2(\mathfrak{U}, A)$.
\end{definition}

\begin{definition}
[(cf. Definition \ref{definition:BasicVectorBundleOverGmGerbeCech})]
\label{definition:BasicHTorsorOverAGerbeCech}
A \emph{basic $H$-torsor on a $\mathfrak{U}$-$A$-gerbe $\underline{\alpha}$ over $X$} is a collection
\[
\Big( \underline{P} = \{ P_i \}_{i \in I}, \; \underline{h} = \{ h_{ij} \}_{i, j \in I} \Big)
\]
of $H$-torsors $P_i$ on $U_i$, together with isomorphisms $h_{ij}: P_j |_{U_{ij}} \to P_i |_{U_{ij}}$ satisfying $h_{ii} = \operatorname{id}_{P_i}$, $h_{ij} = h_{ji}^{-1}$, and the $\underline{\alpha}$-twisted cocycle condition,
\[
h_{ij} h_{jk} h_{ki} = \alpha_{ijk} \operatorname{id}_{P_i},
\]
on $U_{ijk}$ for any $i, j, k \in I$. In parallel with the case of the $GL_n$, we also use the term \emph{$\alpha$-twisted $H$-torsor on $X$} (cf. Definition \ref{definition:TwistedSheaf}).
\end{definition}

\begin{definition}
[(cf. Definition \ref{definition:FlatGmGerbeCech})]
\label{definition:FlatAGerbeCech}
A \emph{flat $\mathfrak{U}$-$A$-gerbe over $X$} is the datum of a 2-cocycle $\big( \underline{\alpha}, \underline{\omega}, \underline{F} \big) \in \check{\ZZ}^2 \big( \mathfrak{U}, \mathrm{dR}^A_X \big)$, where
\[
\mathrm{dR}^A_X := \Big[ A \xrightarrow{\;\;a \,\mapsto\, a^{-1}da\;\;} \Omega^1_X \otimes \mathfrak{a} \xrightarrow{\quad d \quad} \Omega^2_X \otimes \mathfrak{a} \xrightarrow{\quad d \quad} \cdots \Big].
\]
is the \emph{$A$-de Rham complex of $X$}.
\end{definition}

The map that we denoted $a \mapsto a^{-1}da$ above is perhaps most clearly expressed in differential-geometric terms: it sends a local section $a: U \to A$ to the $1$-form with values in $\mathfrak{a}$ whose value at $x \in U$ corresponds to
\[
T_x X \xrightarrow{\;\;Ta\;\;} T_{a(x)} A \xrightarrow{\;\;(L_{a^{-1}(x)})_\ast\;\;} T_e A \cong \mathfrak{a}.
\]
For $A = \GG_m$ it is just the logarithmic exterior derivative of \eqref{eq:deRhamComplexG_m}, while for $A = \GG_a$ it reduces to the usual exterior derivative.

\begin{definition}
[(cf. Definition \ref{definition:BasicVectorBundleOverFlatGmGerbeCech})]
\label{definition:BasicHTorsorOverFlatAGerbeCech}
A \emph{basic $H$-torsor on a flat $\mathfrak{U}$-$A$-gerbe $\big( \underline{\alpha}, \underline{\omega}, \underline{F} \big)$ over $X$} is a collection $\big( \underline{P}, \underline{\eta}, \underline{h} \big)$ consisting of a basic $H$-torsor $\big( \underline{P}, \underline{h} \big)$ on $\underline{\alpha}$ together with connections $\eta_i$ on $P_i$ satisfying the equations
\[
\eta_i - h_{ij} \eta_j \, h_{ij}^{-1} = \omega_{ij} \quad\text{and}\quad C(\eta_i) = F_i.
\]
\end{definition}

\begin{definition}
[(cf. Definition \ref{definition:HiggsGmGerbeCech})] \label{definition:HiggsAGerbeCech}
A \emph{Higgs $\mathfrak{U}$-$A$-gerbe} over $X$ is the datum of a 2-cocycle $\big( \underline{\alpha}', \underline{\omega}', \underline{F}' \big) \in \check{\ZZ}^2 \big( \mathfrak{U}, \mathrm{Dol}^A_X \big)$, where
\[
\mathrm{Dol}^A_X := \Big[ A \xrightarrow{\quad 0 \quad} \Omega^1_X \otimes \mathfrak{a} \xrightarrow{\quad 0 \quad} \Omega^2_X \otimes \mathfrak{a} \xrightarrow{\quad 0 \quad} \cdots \Big].
\postdisplaypenalty=500
\]
is the \emph{$A$-Dolbeault complex of $X$}.
\end{definition}

\begin{definition}
[(cf. Definition \ref{definition:BasicVectorBundleOverHiggsGmGerbeCech})]
\label{definition:BasicHTorsorOverHiggsAGerbeCech}
A \emph{basic $H$-torsor on a Higgs $\mathfrak{U}$-$A$-gerbe $\big( \underline{\alpha}', \underline{\omega}', \underline{F}' \big)$ over $X$} is a collection $\big( \underline{P}', \underline{\phi}, \underline{h}' \big)$ consisting of a basic $H$-torsor $\big( \underline{P}', \underline{h}' \big)$ on $\underline{\alpha}$ together with Higgs fields $\phi_i$ on $P'_i$ satisfying the equations
\[
\phi_i - h'_{ij} \phi_j \, (h'_{ij})^{-1} = \omega'_{ij} \quad\text{and}\quad C(\phi_i) = F'_i.
\]
\end{definition}

With these definitions, it is easy to write down the analogue of Theorem \ref{theorem:MainTheoremGLnCech} for twisted $H$-torsors. Its validity, however, is constrained by two conditions on the group $H$. The first one is that $H$ be connected (this is a technical condition: see Remark \ref{remark:ConnectednessOfH}). In the algebraic category, we also need to impose that $K$ be reductive (see \S\ref{section:Analytification}). Subject to these, we have the following statement.

\begin{theorem}
\label{theorem:MainTheoremCech}
Let $\big( \underline{\alpha}, \underline{\omega}, \underline{F} \big)$ be a flat $\mathfrak{U}$-$A$-gerbe over $X$. Then there is a Higgs $\mathfrak{U}$-$A$-gerbe, $\big( \underline{\alpha}', \underline{\omega}', \underline{F}' \big)$, over $X$ for which there is a fully faithful functor
\[
\begin{Bmatrix}
\text{Basic $H$-torsors} \\ \text{on } \big( \underline{\alpha}, \underline{\omega}, \underline{F} \big)
\end{Bmatrix}
\lhook\joinrel\longrightarrow
\begin{Bmatrix}
\text{Basic $H$-torsors} \\ \text{on } \big( \underline{\alpha}', \underline{\omega}', \underline{F}' \big)
\end{Bmatrix}
\]
Conversely, given a Higgs $\mathfrak{U}$-$A$-gerbe, $\big( \underline{\alpha}', \underline{\omega}', \underline{F}' \big)$, there exists a flat $\mathfrak{U}$-$A$-gerbe, $\big( \underline{\alpha}, \underline{\omega}, \underline{F} \big)$, for which the same conclusion holds.
\end{theorem}

\subsection{Outlook}
\label{section:Outlook}

\subsubsection{}

For all that Theorems \ref{theorem:MainTheoremGLnCech} and \ref{theorem:MainTheoremCech} do offer twisted versions of the nonabelian Hodge theorem, they are not without several important shortcomings.

\begin{itemize}
\item The most obvious one is that they are stated in terms of choices, not only of a cover $\mathfrak{U}$ of $X$, but also of explicit representatives for all of the objects involved. It is possible ---though rather laborious and utterly unilluminating--- to show by brute force that the statements are indeed independent of such choices. However, their explicit dependence becomes truly problematic when trying to provide proofs based on the classical nonabelian Hodge theorem, for the latter requires the compactness assumption on the base manifold; that is, we cannot just break up these twisted correspondences into local pieces.
\item Somewhat related is the issue of characterizing the essential image of the functors: in these local formulations, it is unclear how to do so. Once again, the only hope comes from trying to use Simpson's theorem.
\item On a different note, the decision to allow the locally defined connections and Higgs fields to have central curvature and to differ by something central on double intersections seems arbitrary at best.
\end{itemize}

A significant portion of our proofs of Theorems \ref{theorem:MainTheoremGLnCech} and \ref{theorem:MainTheoremCech} consists of finding manifestly cover- and cocycle-independent versions of them (Theorems \ref{theorem:MainTheoremGLn} and \ref{theorem:MainTheorem}) that not only rid us of choices but also reveal the naturality of considering central twistings.

\subsubsection{Organization of this paper.}

In \S\ref{section:Gerbes} we look at a particular class of (1-)stacks : $A$-banded gerbes. These provide a geometric realization of degree 2 cohomology classes with values in $A$, and twisted $H$-torsors on their base can be thought of as honest $H$-torsors on them (see \S\ref{section:TorsorsOverGerbes}).

The setting of classical (1-)stacks \cite{MR0262240} is, however, not enough for our purposes. The main limitation of the definition of $A$-banded gerbes in \S\ref{section:Gerbes} is that it still depends on the existence of a cover of the base. As such, it does not work over nongeometric bases ---which we need in order to define flat and Higgs gerbes. The existence of a classifying (2-)stack for $A$-banded gerbes in the theory of principal $\infty$-bundles \cite{arXiv1207.0248} of \S\ref{section:GerbesAsPrincipalInftyBundles} overcomes this difficulty while providing us with a host of powerful techniques that enable us to reduce statements about twisted torsors to simpler ones about untwisted torsors and gerbes.

After a technical interlude (\S\ref{section:1LocalicInftyTopoi}), we recall (\S\ref{section:CohesiveStructures}) the theory of the de Rham construction \cite{ST_deRhamInftyStacks} from the perspective of $\infty$-stacks. In \S\ref{section:CaseOfSmoothProjectiveVariety} we restate Simpson's nonabelian Hodge theorem in this language, and reinterpret some classical abelian Hodge theory results as a Hodge correspondence for gerbes. The passage between the algebraic and the analytic worlds is the object of 
\S\ref{section:Analytification}.

In \S\ref{section:TorsionPhenomena} we go back to the original case of twisted vector bundles, now seen in the light of all the tools developed in previous sections. The failure of the obvious strategy of proof ---as well as how it needs to be modified--- is already visible in this simplest example. After explaining in what generality we hope to be able to prove our result and why (\S\ref{section:DisgressionOnAlgebraicGroups}), we take a step back and collect together the exact statements we will prove in \S\ref{section:Proof}, as well as the assumptions they rely on.

\subsubsection{An important remark about topologies.}

At the beginning of this introduction we mentioned that we work either in the analytic or the {\'e}tale topologies. More precisely, throughout this paper we consider sheaves, stacks and $\infty$-stacks over any one of the following Grothendieck sites:
\begin{itemize}
\item $(\mathsf{Aff}_\CC, \text{{\'e}t})$: the category of affine complex schemes equipped with the {\'e}tale topology, or
\item $(\mathsf{An}, \text{{\'e}t})$: the site of complex analytic spaces endowed with the topology in which covers are jointly surjective collections of local isomorphisms ---also known as the analytic {\'e}tale topology.
\end{itemize}
In a couple of sections (\S\ref{section:CohesiveStructures} and \S\ref{section:Analytification}) we also need the auxiliary site
\begin{itemize}
\item $(\mathsf{Aff}_{\CC, \mathrm{ft}}, \text{{\'e}t})$: the full subsite of $(\mathsf{Aff}_\CC, \text{{\'e}t})$ consisting of affine schemes of finite type over $\operatorname{Spec}\CC$ with the induced topology.
\end{itemize}

Many of our arguments are of a homotopical character: they deal with the formal structure of the $\infty$-topoi in which the objects involved live in. As such, they are rather topology-agnostic: about the only fact that is explicitly dependent on the sites detailed above is that
\[
1 \to A \to H \to K \to 1
\]
is a short exact sequence of (0-truncated $\infty$-)group objects, with $A \subset Z(H)$ (see \S\ref{section:LinearAlgebraicGroups}) ---a fact we already mentioned at the beginning of \S\ref{section:GroupsOtherThanGLn}.

The reader should keep in mind that all of our arguments and results are to be understood to hold in both topologies of interest unless otherwise stated. In particular, \S\ref{section:GeometrizingTwistedTorsors}, \S\ref{section:TorsionPhenomena} and \S\ref{section:Proof} are deliberately vague, while \S\ref{section:HodgeTheory}, \S\ref{section:DisgressionOnAlgebraicGroups} and \S\ref{section:StatementMainTheorems} are explicit in the choice of topology.
 
\section{Geometrizing twisted torsors}
\label{section:GeometrizingTwistedTorsors}

\subsection{Gerbes}
\label{section:Gerbes}

\begin{definition}[\cite{MR0344253}]
Let $\mathfrak{Y}$ be a (1-)stack over $X$. We say that $\mathfrak{Y}$ is a \emph{gerbe over $X$} if it is locally nonempty and locally connected.
\end{definition}

The first of these conditions means that there exists an open cover $\{ U_i \to X \}_{i \in I}$ of $X$ such that the canonical maps $\mathfrak{Y} \times_X U_i \to U_i$ all have global sections, while the second one ensures that we can choose said cover such that the groupoid of global sections of each $\mathfrak{Y} \times_X U_i \to U_i$ contains a single isomorphism class.

\begin{definition}[\cite{MR0344253,MR1086889}]
Let $G$ be a linear algebraic group over $\CC$, and $\mathfrak{Y}$ be a (1-)stack over $X$. We say that $\mathfrak{Y}$ is a \emph{$G$-gerbe over $X$} if it is locally isomorphic to $BG \times X$.
\end{definition}

Equivalently, we have $\mathfrak{Y} \times_X U_i \simeq BG \times U_i$ over $U_i$ for each $i \in I$ in a suitable cover. As a sanity check, note that a $G$-gerbe is a gerbe.

The attentive reader might have noticed that the concept of a $G$-gerbe bears a striking similarity to that of a fiber bundle, only now both the fiber, $BG$, and the structure group, $\underline{\operatorname{Aut}}(BG)$, are bona fide stacks ---rather than $0$-truncated objects. We will formalize this thought in \S\ref{section:GerbesAsPrincipalInftyBundles} using the notion of $\infty$-bundles \cite{arXiv1207.0248}.

However, we can manage with the theory of crossed modules \cite{MR0017537} to prove that $G$-gerbes are classified by $H^1(X, \underline{\operatorname{Aut}}(BG))$ \cite{MR1086889}. Here the automorphism stack of $BG$ ---which goes also by the name of the automorphism 2-group of $G$ \cite{MR2068521}---  is represented by the crossed module $G \to \operatorname{Aut}(G)$, and $H^1$ refers to crossed module cohomology. The exact sequence of crossed modules,
\begin{equation} \label{eq:diagramCrossedModules}
1 \longrightarrow \Big[ G \to \operatorname{Inn}(G) \Big] \longrightarrow  \Big[ G \to \operatorname{Aut}(G) \Big] \longrightarrow  \Big[ 1 \to \operatorname{Out}(G) \Big] \longrightarrow 1,
\end{equation}
induces an exact sequence of pointed sets,
\begin{equation} \label{eq:diagramBand}
H^1(X, G \to \operatorname{Inn}(G)) \longrightarrow H^1(X, G \to \operatorname{Aut}(G)) \xrightarrow{\;\;\beta\;\;} H^1(X, \operatorname{Out}(G)).
\end{equation}
Given a $G$-gerbe $\mathfrak{Y}$ over $X$, the $\operatorname{Out}(G)$-torsor classified by $\beta([\mathfrak{Y}])$ is called the \emph{$G$-band of $\mathfrak{Y}$}.

\begin{definition}
\label{definition:bandedGerbes}
A $G$-gerbe over $X$ is called a \emph{$G$-banded gerbe} if its $G$-band is the trivial $\operatorname{Out}(G)$-torsor on $X$.
\end{definition}

In case $G = A$ is \emph{abelian}, the diagram \eqref{eq:diagramCrossedModules} simplifies to
\[
1 \longrightarrow \Big[ A \to 1 \Big] \longrightarrow  \Big[ A \to \operatorname{Aut}(A) \Big] \longrightarrow  \Big[ 1 \to \operatorname{Aut}(A) \Big] \longrightarrow 1
\]
and the sequence \eqref{eq:diagramBand} yields
\[
0 \longrightarrow H^2(X, A) \longrightarrow H^1(X, A \to \operatorname{Aut}(A)) \xrightarrow{\;\;\beta\;\;} H^1(X, \operatorname{Aut}(A)).
\]

\begin{proposition}
\label{proposition:ClassificationBandedAGerbes}
$A$-banded gerbes over $X$ are classified by $H^2(X, A)$.
\end{proposition}

In the remainder we will only deal with $A$-banded gerbes, so we will abuse terminology and call them simply $A$-gerbes, or even just gerbes if the group $A$ is clear from the context.

\subsection{Torsors on gerbes}
\label{section:TorsorsOverGerbes}

\subsubsection{Presentations of gerbes.}
\label{section:PresentationsOfGerbes}

We briefly recall here a presentation of $A$-gerbes by gluings of the local pieces $BA \times U_i$. Much of the material and the notation in this section is borrowed from \cite{MR2399730}, which the reader is encouraged to consult for an extended exposition.

Given a class $\alpha \in H^2(X, A)$, denote by $_\alpha X$ the $A$-gerbe that it classifies\footnote{More precisely, take $_\alpha X$ to be any $A$-gerbe whose equivalence class is given by $\alpha$.}. With the choice of a cover $\mathfrak{U} = \{ U_i \to X \}_{i \in I}$ of $X$ and a representing cocycle $\underline{\alpha} = \{ \alpha_{ijk} \}_{i, j, k \in I} \in \check{Z}^2(\mathfrak{U}, A)$, we have compatible groupoid presentations of $X$,
\begin{equation} \label{eq:GroupoidPresentationX}
\bigg(
  \begin{tikzpicture}[x=100pt,baseline=(1.base)]
  \node (1) at (0,1) {$\mathfrak{R} := \bigsqcup_{i, j \in I} U_{ij}$};
  \node (2) at (1,1) {$U := \bigsqcup_{i \in I} U_i$};
  \path[semithick,->]
    ([yshift=4pt]1.east) edge node[pos=0.5,scale=0.75,descr,yshift=1pt] {$\boldsymbol{\mathfrak{s}}$} ([yshift=4pt]2.west)
    ([yshift=-4pt]1.east) edge node[pos=0.5,scale=0.75,descr] {$\boldsymbol{\mathfrak{t}}$} ([yshift=-4pt]2.west);
\end{tikzpicture},
\boldsymbol{\mathfrak{m}}, \boldsymbol{\mathfrak{i}}, \boldsymbol{\mathfrak{e}} \bigg),
\end{equation}
and of $_\alpha X$,
\begin{equation} \label{eq:GroupoidPresentationAlphaX}
\bigg(
  \begin{tikzpicture}[x=110pt,baseline=(1.base)]
  \node (1) at (0,1) {$R := \bigsqcup_{i, j \in I} U_{ij} \times A$};
  \node (2) at (1,1) {$U = \bigsqcup_{i \in I} U_i$};
  \path[semithick,->]
    ([yshift=4pt]1.east) edge node[pos=0.45,scale=0.75,descr,yshift=1pt] {$\boldsymbol{s}$} ([yshift=4pt]2.west)
    ([yshift=-4pt]1.east) edge node[pos=0.45,scale=0.75,descr] {$\boldsymbol{t}$} ([yshift=-4pt]2.west);
\end{tikzpicture},
\boldsymbol{m}, \boldsymbol{i}, \boldsymbol{e} \bigg).
\end{equation}
The maps in \eqref{eq:GroupoidPresentationX} are the obvious ones. Those in the presentation \eqref{eq:GroupoidPresentationAlphaX} of $_\alpha X$ are as follows.
\begin{itemize}
\item The source and target morphisms come from those of \eqref{eq:GroupoidPresentationX} and the canonical projection $\pi: R \to \mathfrak{R}$ as $\boldsymbol{s} = \boldsymbol{\mathfrak{s}} \circ \pi$ and $\boldsymbol{t} = \boldsymbol{\mathfrak{t}} \circ \pi$.
\item The identity and the inverse are given by
\[
\begin{tikzpicture}[x=80pt]
  \node (11) at (0,1) {$U$};
  \node (12) at (1,1) {$R$};
  \node (21) at (0,0) {$(x, i)$};
  \node (22) at (1,0) {$((x, i, i), 1_A)$};
  \path[semithick,->]
    (11) edge node[pos=0.5,scale=0.75,yshift=8pt] {$\boldsymbol{e}$} (12);
  \path[semithick,|->]
    (21) edge (22);
\end{tikzpicture}
\qquad\qquad
\begin{tikzpicture}[x=90pt]
  \node (11) at (0,1) {$R$};
  \node (12) at (1,1) {$R$};
  \node (21) at (0,0) {$((x,i,j), a)$};
  \node (22) at (1,0) {$((x,j,i), a^{-1})$};
  \path[semithick,->]
    (11) edge node[pos=0.5,scale=0.75,yshift=8pt] {$\boldsymbol{i}$} (12);
  \path[semithick,|->]
    (21) edge (22);
\end{tikzpicture}
\]
\item Finally, the multiplication $\boldsymbol{m}$ is the one encoding the class of the gerbe:
\[
\begin{tikzpicture}[x=150pt]
  \node (11) at (0,1) {$R \times_{\boldsymbol{t}, U, \boldsymbol{s}} R$};
  \node (12) at (1,1) {$R$};
  \node (21) at (0,0) {$((x, i, j), a) \times ((x, j, k), b)$};
  \node (22) at (1,0) {$((x, i, k), \alpha_{ijk}(x)ab)$};
  \path[semithick,->]
    (11) edge node[pos=0.5,scale=0.75,yshift=8pt] {$\boldsymbol{m}$} (12);
  \path[semithick,|->]
    (21) edge (22);
\end{tikzpicture}
\]
\end{itemize}
We denote by $\boldsymbol{p}_1$ and $\boldsymbol{p}_2$ the canonical projections $R \times_{\boldsymbol{t}, U, \boldsymbol{s}} R \rightrightarrows R$ onto the first and second factors, respectively.

\subsubsection{}

Given the groupoid presentation \eqref{eq:GroupoidPresentationAlphaX} of $_\alpha X$, descent theory tells us that an $H$-torsor on $_\alpha X$ is given by an $H$-torsor $P \to U$ together with an isomorphism $\boldsymbol{j}: \boldsymbol{t}^\ast P \longrightarrow \boldsymbol{s}^\ast P$ rendering commutative the following diagram:
\begin{equation} \label{eq:CocycleCondition}
\begin{gathered}
\begin{tikzpicture}[x=210pt,y=80pt]
  \node(1) at (0.3,1) {$\boldsymbol{p}_1^\ast \boldsymbol{s}^\ast P$};
  \node(2) at (0.7,1) {$\boldsymbol{p}_1^\ast \boldsymbol{t}^\ast P$};
  \node(3) at (1,0.5) {$\boldsymbol{p}_2^\ast \boldsymbol{s}^\ast P$};
  \node(4) at (0.7,0) {$\boldsymbol{p}_2^\ast \boldsymbol{t}^\ast P$};
  \node(5) at (0.3,0) {$\boldsymbol{m}^\ast \boldsymbol{t}^\ast P$};
  \node(6) at (0,0.5) {$\boldsymbol{m}^\ast \boldsymbol{s}^\ast P$};
  \draw[double distance=1.5pt] (2) -- (3);  
  \draw[double distance=1.5pt] (4) -- (5);  
  \draw[double distance=1.5pt] (6) -- (1);
  \path[semithick,->]
    (2) edge node[pos=0.45,scale=0.8,yshift=10pt] {$\boldsymbol{p}_1^\ast \boldsymbol{j}$} (1)
    (4) edge node[pos=0.3,scale=0.8,xshift=18pt] {$\boldsymbol{p}_2^\ast \boldsymbol{j}$} (3)
    (5) edge node[pos=0.3,scale=0.8,xshift=-22pt] {$\boldsymbol{m}^\ast \boldsymbol{j}$} (6);
\end{tikzpicture}
\end{gathered}
\end{equation}
The map $\boldsymbol{j}$ breaks up into a collection of isomorphisms of $H$-torsors:
\[
\Big\{ h_{ij} : \boldsymbol{t}^\ast P \big|_{U_{ij} \times A} \longrightarrow \boldsymbol{s}^\ast P \big|_{U_{ij} \times A} \Big\}_{i, j \in I}
\]
These induce isomorphisms $h_{ij}(x, a)$ of the fiber $H$ over each point $(x, a)$ of the base $U_{ij} \times A$ commuting with the tautological right action of $H$ on itself; that is to say, each $h_{ij}(x,a)$ acts as left multiplication by an element of $H$. In terms of these the cocycle condition \eqref{eq:CocycleCondition} results in the equation
\begin{equation} \label{eq:CocycleConditionGauge}
h_{ij}(x, a) h_{jk}(x, b) = h_{ik}(x, \alpha_{ijk}(x)ab)
\end{equation}
for $x \in U_{ijk}$.

Though the $H$-torsors $\boldsymbol{s}^\ast P = \pi^\ast \boldsymbol{\mathfrak{s}}^\ast P$ and $\boldsymbol{t}^\ast P = \pi^\ast \boldsymbol{\mathfrak{t}}^\ast P$ obviously descend to $\mathfrak{R}$, there is in general no hope for $\boldsymbol{j}$ to do so too unless we impose some additional constraint on it. Since we are performing descent along the $A$-torsor $\pi: R \to \mathfrak{R}$, it is clear that we need to equip $\boldsymbol{s}^\ast P$ and $\boldsymbol{t}^\ast P$ with $A$-equivariant structures ---lifts of the $A$-action on $R$ to the total spaces of these torsors--- and ask for $\boldsymbol{j}$ to be equivariant with respect to them. There are two natural choices for such a lift coming from the trivial and tautological maps from $A$ to $H$. Indeed, since $\boldsymbol{s}^\ast P$ and $\boldsymbol{t}^\ast P$ are $H$-torsors, they come equipped with a right action of $H$ preserving the fibers of their respective projections to $R$. Composing this action with the two maps above yields fiberwise actions of $A$ on $\boldsymbol{s}^\ast P$ and $\boldsymbol{t}^\ast P$, respectively. Collating these with the $A$-action on the base $R$ yields the required $A$-equivariant structures on these torsors.

\begin{definition} \label{definition:BasicHTorsorsOverGerbes}
We say that an $H$-torsor $(P, \boldsymbol{j})$ on $_\alpha X$ is \emph{basic} if $\boldsymbol{j}$ is $A$-equivariant for the trivial and tautological $A$-equivariant structures on $\boldsymbol{t}^\ast P$ and $\boldsymbol{s}^\ast P$, respectively.
\end{definition}

With this choice of $A$-equivariant structures, $\boldsymbol{j}$ is equivariant if
\[
h_{ij}(x, a)b = h_{ij}(x, ab).
\]
Defining
\[
\Big\{ h_{ij} : \boldsymbol{\mathfrak{t}}^\ast P \big|_{U_{ij}} \longrightarrow \boldsymbol{\mathfrak{s}}^\ast P \big|_{U_{ij}} \Big\}_{i, j \in I}
\]
by $h_{ij}(x) := h_{ij}(x, 1)$ (this abuse of notation should not cause any confusion), equation \eqref{eq:CocycleConditionGauge} is equivalent to $h_{ij}h_{jk} = \alpha_{ijk} h_{ik}$, which, in turn, implies ---and is implied by--- the conditions
\begin{equation}
\label{eq:TwistedHbundle}
h_{ii} = 1, \quad h_{ij} = h_{ji}^{-1}, \quad \text{ and } \quad h_{ij}h_{jk}h_{ki} = \alpha_{ijk}.
\end{equation}
The following statement is now clear.

\begin{proposition} \label{proposition:BasicHTorsorsOverGerbes}
The category of $\alpha$-twisted $H$-torsors on $X$ is equivalent to the category of basic\ $H$-torsors on the $A$-gerbe $_\alpha X$.
\end{proposition}

\subsubsection{Obstruction gerbes.}
\label{section:ObstructionGerbes}

Given a $K$-torsor $Q \to X$ we can ask whether there exists an $H$-torsor giving rise to it through the obvious map in the long exact sequence of cohomology
\begin{equation}
\label{eq:ObstructionAGerbe}
H^1(X, A) \longrightarrow H^1(X, H) \longrightarrow H^1(X, K) \xrightarrow{\;\;ob_A\;\;} H^2(X, A).
\end{equation}
It is clear that a necessary and sufficient condition is for the $A$-gerbe classified by $ob_A([Q \to X])$ to be trivial; i.e., globally equivalent to $BA \times X$. We call the latter the \emph{obstruction $A$-gerbe} of the $K$-torsor $Q \to X$. If it is not trivial, we can only lift $Q \to X$ to an $ob([Q \to X])$-twisted $H$-torsor on $X$ ---this is the reason why some authors prefer the term \emph{lifting gerbe}.

To show this, choose a cover $\mathfrak{U} = \{ U_i \to X \}_{i \in I}$ of $X$ that trivializes $Q \to X$, so that the $K$-torsor is specified solely by the transition functions $\underline{k} = \{ k_{ij} \}_{i, j \in I} \in \check{Z}^1(\mathfrak{U}, K)$. Over each $U_i$, our $K$-torsor is trivial and hence liftable to the trivial $H$-torsor on $U_i$. Over each double intersection $U_{ij}$ we can lift the transition function $k_{ij}$ to an element $h_{ij} \in \Gamma(U_{ij}, H)$. But in general we cannot make $\underline{h} = \{ h_{ij} \} \in \check{C}^1(\mathfrak{U}, H)$ into a 1-cocycle; rather,
\[
\underline{a} = \Big\{ a_{ijk} := h_{ij}h_{jk}h_{ki} \Big\}_{i, j, k \in I} \in \check{Z}^2(\mathfrak{U}, A)
\]
provides a representative for the class of the obstruction gerbe of $Q \to X$. Different choices of lifts $h_{ij}$ produce different representatives for the class $[\underline{a}] \in \check{H}^2(\mathfrak{U}, A)$.

On the other hand, notice (cf. \S\ref{section:RemarksAboutTorsionAndProjectivization}) that an $\alpha$-twisted $H$-torsor $S \to X$ induces an untwisted $K$-torsor on $X$ by change of fiber: namely, $S \times_H K \to X$.

\subsection{Gerbes as principal \texorpdfstring{$\infty$}{infinity}-bundles}
\label{section:GerbesAsPrincipalInftyBundles}

\subsubsection{}

In this section, we up the homotopical ante and place ourselves squarely in the world of $\infty$-topoi (in essence, $\infty$-categories of $\infty$-stacks on some Grothendieck site). The reader unfamiliar with these objects is strongly encouraged to peruse the literature before continuing. For $\infty$-categories and $\infty$-topoi, the canonical reference has come to be \cite{MR2522659} (the point of view of \cite{MR2137288,MR2394633} might be more appealing to algebraic geometers, but we stick with Lurie's terminology). Those wanting an introduction to $\infty$-categories instead of an encyclopedic treatise can read \cite{arXiv1007.2925}. A short account, particularly tailored to our purposes, of what an $\infty$-topos is can be found in \cite{arXiv1207.0248}, where the notion of a principal $\infty$-bundle first appeared.

One of the most important insights of \cite{arXiv1207.0248} is that the classical notion of principal bundle find its most natural home in the context of $\infty$-topoi. The classical definition requires not only an action of the structure group on the total space of the bundle, but also the freeness of said action and a local triviality condition. The authors of \emph{loc.cit.} explain how this freeness condition is an artifice: it is only necessary if we insist on the base of the bundle being a space ---that is, a $0$-truncated object of the appropriate $\infty$-topos. The local triviality, on the other hand, is tautological when understood in a generalized sense that is natural from the point of view of $\infty$-topoi ---the magic words here being \emph{effective epimorphism}.

But once we are willing to accept stacks as bases for these bundles, we should also admit them as total spaces, and even as structure groups.

\begin{definition}[{\cite[Definition 3.4]{arXiv1207.0248}}]
Let $\mathbf{H}$ be an $\infty$-topos, $G$ an $\infty$-group object in it, and $X \in \mathbf{H}$. A \emph{$G$-principal $\infty$-bundle over $X$} is an object $(P \to X) \in \mathbf{H}_{/X}$ equipped with a $G$-action such that $P \to X$ exhibits the quotient $X \simeq P/\!\!/G$.
\end{definition}

The $\infty$-category of $G$-principal $\infty$-bundles over $X$ is, in fact, an $\infty$-groupoid, denoted $G\mathrm{Bund}(X)$. The following result states that the \emph{delooping} $BG$ of $G$ classifies these objects.

\begin{theorem}[{\cite[Theorem 3.19]{arXiv1207.0248}}]
\label{theorem:ClassificationPrincipalInftyBundles}
For all $X \in \mathbf{H}$ and every $\infty$-group $G$, there is a natural equivalence of $\infty$-groupoids
\[
G\mathrm{Bund}(X) \simeq \mathbf{Map}_\mathbf{H}(X, BG)
\]
which, on vertices, maps a bundle $P \to X$ to a morphism $X \to BG$ (its \emph{classifying morphism}, denoted $[P \to X]$) for which $P \to X \to BG$ is a fiber sequence.
\end{theorem}

Although they do not make it explicit, the authors of \cite{arXiv1207.0248} do define the $\infty$-category $G\mathrm{Bund}$ of $G$-principal $\infty$-bundles over arbitrary bases. Morphisms between $G$-principal $\infty$-bundles over different bases are defined in the obvious way: if $f \in \mathbf{Map}_\mathbf{H}(X,Y)$, and $P \to X$ and $Q \to Y$ are $G$-principal $\infty$-bundles over $X$ and $Y$, respectively, we can pull back $Q \to Y$ to $X$ by $f$, and consider morphisms over $X$ between $P \to X$ and $Q \times^h_{Y,f}X \to X$.

\begin{corollary}
For every $\infty$-group $G$, there is a natural equivalence of $\infty$-categories
\[
G\mathrm{Bund} \simeq \mathbf{H}_{/BG}
\]
\end{corollary}

The description of mapping spaces in overcategories of \cite[Lemma~5.5.5.12]{MR2522659} then yields the following characterization of the mapping spaces in $G\mathrm{Bund}$.

\begin{lemma}
\label{lemma:MorphismsBetweenPrincipalInftyBundles}
For $G$ an $\infty$-group, and any two $G$-principal $\infty$-bundles $P \to X$ and $Q \to Y$, there are natural equivalences of $\infty$-groupoids
\begin{align*}
\Map_{G\mathrm{Bund}}\!\big(P\to X, Q \to Y\big)
&\simeq \Map_{\mathbf{H}_{/BG}}\!\big([P \to X], [Q \to Y]\big) \\
&\simeq \lim \left\{
\begin{tikzpicture}[x=150pt,y=35pt,baseline={([yshift=-5pt]current bounding box.west)}]
  \node (12) at (1,1) {$\ast$};
  \node (21) at (0,0) {$\Map_\mathbf{H}\!\big(X, Y\big)$};
  \node (22) at (1,0) {$\Map_\mathbf{H}\!\big(X, BG\big)$};
  \path[semithick,->]
    (12) edge[shorten <=-2pt] node[pos=0.4,scale=0.75,xshift=26pt] {$[P \to X]$} (22)
    (21) edge node[pos=0.5,scale=0.75,yshift=10pt] {$[Q \to Y] \circ -$} (22);
\end{tikzpicture}
\right\}
\end{align*}
\end{lemma}

To simplify notation, we write $\mathbf{Map}(-,-)$ for the mapping space $\mathbf{Map}_\mathbf{H}(-,-)$ whenever the $\infty$-topos $\mathbf{H}$ is clear from the context. We also denote a $G$-principal $\infty$-bundle by its total space whenever no confusion could arise about what the $G$-action is, and use $\mathbf{Map}_G(-,-)$ for the mapping space $\mathbf{Map}_{G\mathrm{Bund}}(-,-)$. In these terms, the conclusion of Lemma \ref{lemma:MorphismsBetweenPrincipalInftyBundles} reads
\[
\Map_G\!\big(P, Q\big) \simeq \lim \left\{
\begin{tikzpicture}[x=120pt,y=35pt,baseline={([yshift=-5pt]current bounding box.west)}]
  \node (12) at (1,1) {$\ast$};
  \node (21) at (0,0) {$\Map\!\big(X, Y\big)$};
  \node (22) at (1,0) {$\Map\!\big(X, BG\big)$};
  \path[semithick,->]
    (12) edge[shorten <=-2pt] node[pos=0.4,scale=0.75,xshift=12pt] {$[P]$} (22)
    (21) edge node[pos=0.5,scale=0.75,yshift=10pt] {$[Q] \circ -$} (22);
\end{tikzpicture}
\right\}
\]
This last bit of notation is meant to evoque the classical case, in direct analogy of which we refer to the elements of $\Map_G\!\big(P, Q\big)$ as \emph{$G$-equivariant} morphisms between $P$ and $Q$.

\subsubsection{}

In the case that occupies us, an abelian algebraic group $A$ gives rise (by delooping) to an $\infty$-group object both in the $\infty$-topos of complex-analytic $\infty$-stacks, and in that of {\' e}tale $\infty$-stacks\footnote{At this point, any sensible notion of the $\infty$-topos of complex-analytic (resp., {\'e}tale) $\infty$-stacks can be used. In \S\ref{section:1LocalicInftyTopoi} we will make a particular choice.}: the classifying stack $BA$ of $A$-torsors ---the group operation being given by convolution of $A$-torsors.
\begin{definition}
\label{definition:AGerbesAsPrincipalInftyBundles}
An \emph{$A$-gerbe over a stack $\mathfrak{X}$} is a $BA$-principal $\infty$-bundle over $\mathfrak{X}$.
\end{definition}

This definition generalizes Definition \ref{definition:bandedGerbes}: it purges it from the choice of a cover of the base, and thus it works over nongeometric bases too. By Theorem \ref{theorem:ClassificationPrincipalInftyBundles}, the $\infty$-groupoid of such objects is given by the mapping space $\Map\!\big(\mathfrak{X}, B^2A\big)$. The latter is a 2-category, with equivalence classes given by $\pi_0 \Map\!\big(\mathfrak{X}, B^2A\big) \cong H^2(\mathfrak{X}, A)$, in accordance with Proposition \ref{proposition:ClassificationBandedAGerbes}. Following the notation of \S\ref{section:PresentationsOfGerbes}, we denote by $_\alpha \mathfrak{X}$ the (any) $A$-gerbe over $\mathfrak{X}$ classified by $\alpha \in \Map\!\big(\mathfrak{X}, B^2A\big)_0$\footnote{Sometimes we will abuse language by conflating the two, saying that $\alpha$ itself is an $A$-gerbe.}.

\subsubsection{}
\label{section:BasicTorsorsOverGerbes}

Consider now the exact sequence of sheaves $1 \to A \to H \to K \to 1$. It gives rise to a long fiber sequence of stacks ---the \emph{Puppe sequence}---,
\[
1 \to A \to H \to K \to BA \to BH \to BK \xrightarrow{\;ob_A\;} B^2A,
\]
that exhibits $BH$, the classifying stack of $H$-torsors, as an $A$-gerbe over $BK$ ---and induces \eqref{eq:ObstructionAGerbe} after passage to cohomology.

\begin{definition}
[(cf. Definition \ref{definition:BasicHTorsorsOverGerbes})]
\label{definition:BasicHTorsorOverGerbe}
\hspace{-5pt}\footnote{A definition in the same spirit appeared in \cite{MR2966944}.}
A basic\ $H$-torsor on an $A$-gerbe ${}_{\alpha}\mathfrak{X}$ is an $H$-torsor on ${}_{\alpha}\mathfrak{X}$ whose classifying morphism is $BA$-equivariant. The category of such objects is the mapping space
\[
\Map_{BA}\!\big( {}_{\alpha}\mathfrak{X}, BH \big)
\]
\end{definition}

A direct application of Lemma \ref{lemma:MorphismsBetweenPrincipalInftyBundles} provides a presentation of this (1-)category as a limit:
\begin{equation} \label{eq:BasicHTorsors}
\Map_{BA}\!\big( {}_{\alpha}\mathfrak{X}, BH \big) \simeq \lim \left\{
\begin{tikzpicture}[x=110pt,y=35pt,baseline={([yshift=-5pt]current bounding box.west)}]
  \node (12) at (1,1) {$\ast$};
  \node (21) at (0,0) {$\Map\!\big(\mathfrak{X}, BK\big)$};
  \node (22) at (1,0) {$\Map\!\big(\mathfrak{X}, B^2 A\big)$};
  \path[semithick,->]
    (12) edge[shorten <=-2pt] node[pos=0.4,scale=0.75,xshift=8pt] {$\alpha$} (22)
    (21) edge node[pos=0.5,scale=0.75,yshift=8pt] {$ob_A$} (22);
\end{tikzpicture}
\right\}
\end{equation}
In words: a basic $H$-torsor on $_\alpha \mathfrak{X}$ is given by a $K$-torsor on $\mathfrak{X}$ together with an equivalence between the obstruction $A$-gerbe of the latter and $_\alpha \mathfrak{X}$ (cf. \S\ref{section:RemarksAboutTorsionAndProjectivization}).

A remark is in order: in higher homotopical situations, equivariance is usually extra structure. In the case of Definition \ref{definition:BasicHTorsorOverGerbe}, however, it is truly just a condition, since both $BA$ and $BH$ are 1-truncated $\infty$ stacks ---in agreement with Definition \ref{definition:BasicHTorsorsOverGerbes}.

\section{Hodge theory}
\label{section:HodgeTheory}

\subsection{1-localic \texorpdfstring{$\infty$}{infinity}-topoi and their hypercompletions}
\label{section:1LocalicInftyTopoi}

\subsubsection{}

Let $(\mathcal{C}, \tau)$ be a Grothendieck site admitting finite limits. There are two closely related variants of the $\infty$-topos of $\infty$-stacks over it:
\begin{itemize}
\item We denote by $\mathsf{St}(\mathcal{C}, \tau)^\wedge$ the hypercomplete $\infty$-topos presented by both the local injective \cite{MR906403} and local projective \cite{MR1876801,MR1870515} model structures on the category $[\mathcal{C}^\mathrm{op}, \mathsf{sSet}]$ of simplicial presheaves on $\mathcal{C}$. These can also be realized as left Bousfield localizations at the collection of all hypercovers of the global injective and global projective model structures, respectively, on $[\mathcal{C}^\mathrm{op}, \mathsf{sSet}]$ \cite{MR2034012}.
\item Localizing any of the global model structures at the smaller class of {\v C}ech hypercovers yields the so-called \emph{1-localic} $\infty$-topos of $\infty$-stacks on $(\mathcal{C}, \tau)$ \cite[Definition 6.4.5.8]{MR2522659}, which we simply denote by $\mathsf{St}(\mathcal{C}, \tau)$.
\end{itemize}

From their descriptions as left Bousfield localizations it should come as no surprise that there is a geometric morphism of $\infty$-topoi,
\[
\begin{tikzpicture}[x=80pt]
  \node (1) at (0,0.5) {$\mathsf{St}(\mathcal{C}, \tau)^\wedge$};
  \node (2) at (1,0.5) {$\mathsf{St}(\mathcal{C}, \tau)$};
  \path[semithick,{Hooks[right]}->]
    ([yshift=-4pt]1.east) edge ([yshift=-4pt]2.west);
  \path[semithick,->]
    ([yshift=4pt]2.west) edge ([yshift=4pt]1.east);
\end{tikzpicture},
\]
in which the right adjoint is fully faithful. This exhibits $\mathsf{St}(\mathcal{C}, \tau)^\wedge$ as a reflective sub-$\infty$-category of $\mathsf{St}(\mathcal{C}, \tau)$, and is referred to as its \emph{hypercompletion} (see the discussion above \cite[Lemma 6.5.2.9]{MR2522659}). Its objects ---which are then said to be \emph{hypercomplete}--- can be described as those $\infty$-stacks that satisfy descent with respect to all hypercovers, as opposed to only with respect to {\v C}ech hypercovers. \cite[Theorem 6.5.3.12]{MR2522659}. This difference between descent and hyperdescent disappears if we only work with truncated objects ---that is, $\infty$-stacks whose homotopy sheaves vanish above some finite level---; indeed, these are always hypercomplete \cite[Lemma 6.5.2.9]{MR2522659}. Notice too that this implies that the 1-topoi of 0-truncated objects of these two $\infty$-topoi coincide: they can be realized as (the nerve of) the classical 1-topos of sheaves of sets on $(\mathcal{C}, \tau)$:
\[
\tau_{\leq 0} \mathsf{St}(\mathcal{C}, \tau)^\wedge \simeq \tau_{\leq 0}\mathsf{St}(\mathcal{C}, \tau) \simeq \mathsf{Sh}(\mathcal{C}, \tau).
\]

\subsubsection{}

The key technical advantage of the 1-localic version over its hypercompletion lies in the fact that geometric morphisms are determined by their 0-truncations. 
The concrete statement, letting $\mathrm{Fun}_\ast(\mathscr{X}, \mathscr{Y})$ denote the $\infty$-category of geometric morphisms between two $\infty$-topoi $\mathscr{X}$ and $\mathscr{Y}$ (with the right adjoint mapping $\mathscr{X}$ to $\mathscr{Y}$), is the following.

\begin{proposition}[{\cite[Lemma~6.4.5.6]{MR2522659}}] \label{proposition:LiftingGeometricMorphisms}
For any $\infty$-topos $\mathscr{X}$, restriction induces an equivalence
\[
\mathrm{Fun}_\ast \!\left( \mathscr{X}, \mathsf{St}(\mathcal{C}, \tau) \right) \longrightarrow \mathrm{Fun}_\ast \!\left( \tau_{\leq 0}\mathscr{X}, \tau_{\leq 0}\mathsf{St}(\mathcal{C}, \tau) \right).
\]
\end{proposition}

In the succeeding sections, we will encounter several sites together with functors between them that are classically known to induce geometric morphisms between their associated 1-topoi of sheaves of sets ---they are either continuous or cocontinuous in the terminology of \cite{MR0354652}. We will then immediately be able to lift this geometric morphisms to the 1-localic $\infty$-topoi of $\infty$-stacks over them.

The reader that feels more comfortable working in the hypercomplete $\infty$-topos $\mathsf{St}(\mathcal{C}, \tau)^\wedge$ can breathe a sigh of relief knowing that these two $\infty$-topoi coincide when the 1-topos of their 0-truncated objects has enough points ---which is certainly the case for the sites that we examine below.

\subsection{Towards cohesive structures}
\label{section:CohesiveStructures}

In this section we recall the theory of the de Rham construction for $\infty$-stacks of \cite{ST_deRhamInftyStacks}. Almost all of the material is contained in \emph{loc.cit.}. Perhaps the only novelty resides in the pervasive use of the language of $\infty$-categories. On the one hand, this streamlines certains aspects of the theory; on the other, it allows us to state all of the results at the $\infty$-categorical level and not only at the level of their homotopy categories.

\subsubsection{}
\label{section:Reduction}

Let $(\mathcal{C}, \tau)$ be any one of the following Grothendieck sites:
\begin{itemize}
\item $(\mathsf{Aff}_\CC, \text{{\'e}t})$: the category of affine complex schemes equipped with the {\'e}tale topology;
\item $(\mathsf{Aff}_{\CC, \mathrm{ft}}, \text{{\'e}t})$: the full subcategory of the above consisting of affine schemes of finite type over $\operatorname{Spec}\CC$ with the induced topology; or
\item $(\mathsf{An}, \text{{\'e}t})$: the site of complex analytic spaces endowed with the topology in which covers are jointly surjective collections of local isomorphisms ---also known as the analytic {\'e}tale topology.
\end{itemize}
Of course these three sites are intimately related: the first two in the obvious fashion, and the last two through the analytification functor ---of which we will say more in \S\ref{section:Analytification}. In this section we generically refer to an object in any of these categories as a \emph{representable space}.

For any of the choices above, let $\mathcal{C}_\mathrm{red}$ be the full subcategory of $\mathcal{C}$ consisting of geometrically reduced representable spaces ---made into a site by giving it the induced topology. The inclusion functor $j$ possesses a right adjoint, $\mathrm{red}$, that exhibits the first as a coreflective subcategory of the second:
\begin{equation} \label{eq:ReductionFunctor}
\begin{gathered}
\begin{tikzpicture}[x=60pt]
  \node (1) at (0,0.5) {$\mathcal{C}_{\text{red}}$};
  \node (2) at (1,0.5) {$\mathcal{C}$};
  \path[semithick,{Hooks[right]}->]
    ([yshift=-5pt]1.east) edge node[descr,pos=.5,scale=0.75] {$j$} ([yshift=-5pt]2.west);
  \path[semithick,->]
    ([yshift=5pt]2.west) edge node[descr,pos=.5,scale=0.75] {$\mathrm{red}$} ([yshift=5pt]1.east);
\end{tikzpicture}
\end{gathered}
\end{equation}
The usual yoga of functoriality of categories of presheaves \cite[Expos{\' e} I, \S 5]{MR0354652} yields an adjoint quadruple,
\begin{equation} \label{eq:QuadrupleFromReductionFunctor}
j_! \;\;\dashv\;\; j^\ast \cong \mathrm{red}_! \;\;\dashv\;\; j_\ast \cong \mathrm{red}^\ast \;\;\dashv\;\; j^! := \mathrm{red}_\ast,
\end{equation} 
between the categories $\left[ \mathcal{C}^\mathrm{op}, \mathsf{Set} \right]$ and $\left[ \mathcal{C}_\mathrm{red}^\mathrm{op}, \mathsf{Set} \right]$ of presheaves of sets on $\mathcal{C}$ and $\mathcal{C}_\mathrm{red}$, respectively. Since both $j$ and $\mathrm{red}$ are continuous and cocontinuous \cite[Expos{\' e} III]{MR0354652} , these four functors descend to give another adjoint quadruple of functors, this time between their respective categories of sheaves of sets. Furthermore, $\mathcal{C}$ and $\mathcal{C}_\mathrm{red}$ both possess finite limits, and hence we can use Proposition \ref{proposition:LiftingGeometricMorphisms} to lift the last three functors\footnote{Since $j$ does not preserve finite limits, $j_!$ does not either and hence the pair $j_! \dashv j^\ast$ is \emph{not} a geometric morphism.} of \eqref{eq:QuadrupleFromReductionFunctor} to their 1-localic $\infty$-topoi of $\infty$-stacks:
\begin{equation} \label{eq:InfinitesimalCohesion}
\begin{gathered}
\begin{tikzpicture}[x=100pt]
  \node (1) at (0,0.5) {$\mathsf{St}(\mathcal{C}_{\text{red}}, \tau)$};
  \node (2) at (1,0.5) {$\mathsf{St}(\mathcal{C}, \tau)$};
  \path[semithick,{Hooks[right]}->]
    (1.east) edge node[descr,pos=.5,scale=0.75] {$j_\ast$} (2.west);
  \path[semithick,->]
    ([yshift=-10pt]2.west) edge node[descr,pos=.5,scale=0.75] {$j^\ast$} ([yshift=-10pt]1.east)
    ([yshift=10pt]2.west) edge node[descr,pos=.5,scale=0.75] {$j^!$} ([yshift=10pt]1.east);
\end{tikzpicture}
\end{gathered}
\end{equation}

Denote by $\mathrm{Red} = j \circ \mathrm{red}$ the idempotent comonad associated to the pair \eqref{eq:ReductionFunctor}, which sends a representable space to its induced reduced representable subspace. The adjoint quadruple \eqref{eq:QuadrupleFromReductionFunctor} induces an adjoint triple
\[
\mathrm{Red}_! \cong j_! \circ j^\ast \;\;\dashv\;\; \mathrm{Red}^\ast \cong j^\ast \circ j_\ast \;\;\dashv\;\; \mathrm{Red}_\ast \cong j_\ast \circ j^!
\]
of endofunctors on $\left[ \mathcal{C}^\mathrm{op}, \mathsf{Set} \right]$; again, the last two can be lifted all the way to the $\infty$-level:
\begin{equation} \label{eq:deRhamGeometricMorphism}
\begin{gathered}
\begin{tikzpicture}[x=100pt]
  \node (1) at (0,0.5) {$\mathsf{St}(\mathcal{C}, \tau)$};
  \node (2) at (1,0.5) {$\mathsf{St}(\mathcal{C}, \tau)$};
  \path[semithick,->]
    ([yshift=-5pt]2.west) edge node[descr,pos=.5,scale=0.75] {$(-)_\mathrm{dR}$} ([yshift=-5pt]1.east)
    ([yshift=5pt]1.east) edge node[descr,pos=.5,scale=0.75] {$\delta$} ([yshift=5pt]2.west);
\end{tikzpicture}
\end{gathered}
\end{equation}
where, following the notation of \cite[Proposition 3.3]{ST_deRhamInftyStacks}, we denote the functors $\mathrm{Red}^\ast$ and $\mathrm{Red}_\ast$ (at the level of the $\infty$-topoi) by $(-)_\mathrm{dR}$ and $\delta$, respectively. The first of these receives the name of \emph{de Rham functor}, and the image of an $\infty$-stack $\mathfrak{X}$ under it is its \emph{de Rham stack}, $\mathfrak{X}_\mathrm{dR}$.

The counit $\mathrm{Red} \to \mathrm{id}$ of the comonad $\mathrm{Red}$ induces a natural transformation $\mathrm{id} \to (-)_\mathrm{dR}$ of $\infty$-functors. We say that an $\infty$-stack $\mathfrak{X}$ is \emph{formally smooth} if the natural morphism $\mathfrak{X} \to \mathfrak{X}_\mathrm{dR}$ is an effective epimorphism ---that is, if its de Rham stack is (equivalent to) the $\infty$-colimit of the {\v C}ech nerve of $\mathfrak{X} \to \mathfrak{X}_\mathrm{dR}$. If $\mathfrak{X}$ is 0-truncated this definition coincides with the classical one, which is nothing but the infinitesimal lifting property.

\subsubsection{}

Let $\ast$ denote the terminal category. We can extend \eqref{eq:ReductionFunctor} to a diagram
\[
\begin{tikzpicture}[x=80pt, y=60pt]
  \node (1) at (0,1) {$\mathcal{C}_{\text{red}}$};
  \node (2) at (1,1) {$\mathcal{C}$};
  \node (3) at (0.5, 0) {$\ast$};
  \path[semithick,{Hooks[right]}->]
    ([yshift=-5pt]1.east) edge node[descr,pos=.5,scale=0.75] {$j$} ([yshift=-5pt]2.west);
  \path[semithick,->]
    ([yshift=5pt]2.west) edge node[descr,pos=.5,scale=0.75] {$\mathrm{red}$} ([yshift=5pt]1.east)
    ([xshift=-7pt]1.south) edge[shorten <=4pt,shorten >=-6pt] node[descr,pos=.7,scale=0.75] {$\pi_\mathrm{red}$} ([xshift=-7pt]3.north west)
    ([xshift=5pt]3.north west) edge[shorten <=1pt,shorten >=3pt] node[descr,pos=.5,scale=0.75] {$i_\mathrm{red}$} ([xshift=5pt]1.south)
    ([xshift=7pt]2.south) edge[shorten <=4pt,shorten >=-6pt] node[descr,pos=.7,scale=0.75] {$\pi$} ([xshift=7pt]3.north east)
    ([xshift=-5pt]3.north east) edge[shorten <=1pt,shorten >=3pt] node[descr,pos=.5,scale=0.75] {$i$} ([xshift=-5pt]2.south);
\end{tikzpicture}
\]
Here $\pi$ (resp., $\pi_\mathrm{red}$) is the unique functor to $\ast$, and its right adjoint $i$ (resp., $i_\mathrm{red}$) takes the unique object of $\ast$ to the terminal object of $\mathcal{C}$ (resp., $\mathcal{C}_\mathrm{red}$). The following relations are easy to check:
\begin{equation} \label{eq:Relations}
\pi_\mathrm{red} \circ \mathrm{red} = \pi, \quad \pi \circ j = \pi_\mathrm{red}, \quad j \circ i_\mathrm{red} = i, \quad \mathrm{red} \circ i = i_\mathrm{red}.
\end{equation}
As before, we obtain an adjoint quadruple of functors,
\begin{equation} \label{eq:QuadrupleFromInclusionFunctor}
\pi_! \;\;\dashv\;\; \pi^\ast \cong i_! \;\;\dashv\;\; \pi_\ast \cong i^\ast \;\;\dashv\;\; \pi^! := i_\ast
\end{equation}
\begin{equation} \label{eq:QuadrupleFromInclusionFunctorRed}
\text{(resp.,}\qquad \pi_{\mathrm{red},!} \;\;\dashv\;\; \pi_\mathrm{red}^\ast \cong i_{\mathrm{red},!} \;\;\dashv\;\; \pi_{\mathrm{red},\ast} \cong i_\mathrm{red}^\ast \;\;\dashv\;\; \pi_\mathrm{red}^! := i_{\mathrm{red},\ast} \quad\text{)},
\end{equation}
between the appropriate categories of presheaves of sets. Now, $i$ (resp., $i_\mathrm{red}$) is both continuous and cocontinuous and hence the last three functors in \eqref{eq:QuadrupleFromInclusionFunctor} (resp., \eqref{eq:QuadrupleFromInclusionFunctorRed}) descend to the categories of sheaves of sets\footnote{It is easy to get fooled into thinking that $\pi$ (resp., $\pi_\mathrm{red}$) preserves covers. The fact that this is not true stems from the fact that there are empty representable spaces in $\mathcal{C}$ (resp. $\mathcal{C}_\mathrm{red}$). I thank Urs Schreiber for clearing my confusion about this point.}, and then lift via Proposition \ref{proposition:LiftingGeometricMorphisms} to the 1-localic $\infty$-topoi of $\infty$-stacks. From its avatar as $i^\ast$ (resp., $i_\mathrm{red}^\ast$) it is obvious that  $\pi_\ast$ (resp., $\pi_{\mathrm{red},\ast}$) is nothing but the canonical functor of global sections, which we denote by $\Gamma$ (resp., $\Gamma_\mathrm{red}$) following the standard terminology; its left adjoint, $\pi^\ast$ (resp., $\pi_\mathrm{red}^\ast$), is then the extension to $\infty$-stacks of the locally constant sheaf functor, which we denote by $\mathrm{const}$ (resp. $\mathrm{const}_\mathrm{red}$). All in all, we have the following diagram of $\infty$-functors:
\begin{equation} \label{eq:InfinitesimalCohesionPlusCohesion}
\begin{gathered}
\begin{tikzpicture}[x=100pt, y=125pt]
  \node (1) at (0,0.5) {$\mathsf{St}(\mathcal{C}_{\text{red}}, \tau)$};
  \node (2) at (1,0.5) {$\mathsf{St}(\mathcal{C}, \tau)$};
  \node (3) at (0.5, 0) {$\infty\mathsf{Gpd}$};
  \path[semithick,{Hooks[right]}->]
    (1.east) edge node[descr,pos=.5,scale=0.75] {$j_\ast$} (2.west)
    ([xshift=12pt]3.north east) edge[out=45,in=270] node[descr,pos=.4,scale=0.75,xshift=5pt] {$\mathrm{const}$} ([xshift=10pt]2.south)
    ([xshift=-12pt]3.north east) edge[out=45,in=270] node[descr,pos=.53,scale=0.75,xshift=-2pt] {$\pi^!$} ([xshift=-10pt]2.south);
  \path[semithick,{Hooks[left]}->]
    ([xshift=12pt]3.north west) edge[out=135,in=270] node[descr,pos=.55,scale=0.75,xshift=5pt] {$\pi_\mathrm{red}^!$} ([xshift=10pt]1.south)
    ([xshift=-12pt]3.north west) edge[out=135,in=270] node[descr,pos=.4,scale=0.75,xshift=-8pt] {$\mathrm{const}_\mathrm{red}$} ([xshift=-10pt]1.south);
  \path[semithick,->]
    ([yshift=-10pt]2.west) edge node[descr,pos=.5,scale=0.75] {$j^\ast$} ([yshift=-10pt]1.east)
    ([yshift=10pt]2.west) edge node[descr,pos=.5,scale=0.75] {$j^!$} ([yshift=10pt]1.east)
    (1.south) edge[out=270, in=135] node[descr,pos=.35,scale=0.75,xshift=-2pt] {$\Gamma_\mathrm{red}\!\!\!\!\mbox{}$} (3.north west)
    (2.south) edge[out=270, in=45] node[descr,pos=.38,scale=0.75] {$\Gamma$} (3.north east);
\end{tikzpicture}
\end{gathered}
\end{equation}

The equalities \eqref{eq:Relations} yield a natural equivalence of geometric morphisms of $\infty$-topoi,
\begin{equation} \label{eq:NaturalTransformationFromReductionFunctorAndGlobalSectionsFunctor}
j_\ast \circ \mathrm{const}_\mathrm{red} \simeq \mathrm{const} \;\;\dashv\;\; \Gamma \simeq \Gamma_\mathrm{red} \circ j^!,
\end{equation}
which, together with the counit of the adjunction $\mathrm{const}_\mathrm{red} \dashv \Gamma_\mathrm{red}$, implies the existence of a natural transformation
\[
\mathrm{dis} := \mathrm{const} \circ \Gamma \simeq j_\ast \circ \mathrm{const}_\mathrm{red} \circ \Gamma_\mathrm{red} \circ j^! \longrightarrow j_\ast \circ j^! \simeq \delta.
\]
from the functor that associates to an $\infty$-stack the constant $\infty$-stack on its global sections, to the right adjoint to de Rham functor. In the case $(\mathcal{C}, \tau) = (\mathsf{An}, \text{{\'e}t})$, the relationship between these two is extremely simple for stacks that are deloopings of $0$-truncated group objects.

\begin{proposition}[{\cite[Proposition 3.3]{ST_deRhamInftyStacks}}] \label{proposition:Delta}
\mbox{}\\[-15pt]
\begin{itemize}[topsep=-3pt,partopsep=0pt,parsep=2pt,itemsep=2pt]
\item[(a)] For any complex Lie group $G$, $\delta(BG) \simeq B(\mathrm{dis(G)})$.
\item[(b)] For any abelian complex Lie group $A$, $\delta(B^nA) \simeq B^n(\mathrm{dis(A)})$.
\end{itemize}
\end{proposition}

\subsubsection{}

Although the functor $\mathrm{const}$ (resp., $\mathrm{const}_\mathrm{red}$) might not have a further left adjoint ---and even if it does, the latter is never induced by $\pi_!$ (resp., $\pi_{\mathrm{red}, !}$)--- it does preserve finite limits and hence possesses a pro-left adjoint \cite[Expos{\' e} I, \S 8.11]{MR0354652} known as the \emph{fundamental pro-$\infty$-groupoid} functor \cite{NoteEtaleHomotopy} (see also \cite[Appendix A]{HigherAlgebra} and \cite{arXiv_qalg_9609004}):
\[
\Pi_\infty : \mathsf{St}(\mathcal{C}, \tau) \longrightarrow \mathrm{Pro}(\infty\mathsf{Gpd}) \qquad \left( \text{resp.,}\quad \Pi_{\infty\mathrm{red}} : \mathsf{St}(\mathcal{C}_\mathrm{red}, \tau) \longrightarrow \mathrm{Pro}(\infty\mathsf{Gpd}) \quad\right).
\]
We will not go further in the study of this functor. Here we simply note that the natural equivalence \eqref{eq:NaturalTransformationFromReductionFunctorAndGlobalSectionsFunctor} of geometric morphisms implies a natural equivalence of $\infty$-functors,
\begin{equation} \label{eq:FundamentalProInftyGroupoidAndReduction}
\Pi_{\infty\mathrm{red}} \circ j^\ast \simeq \Pi_\infty,
\end{equation}
that we will use below in a cohomology computation.

\subsubsection{}

We close this section with a comment about its title. Inspired by work of Lawvere \cite{MR2125786,MR2369017} in the context of 1-topoi, Schreiber defines a \emph{cohesive $\infty$-topos} \cite[\S 3.4]{arXiv1310.7930} to be an $\infty$-topos $\mathsf{H}$ whose global sections geometric morphism extends to a quadruple
\vspace{1pt}
\begin{equation} \label{eq:DefinitionCohesiveInftyTopos}
\begin{gathered}
\begin{tikzpicture}[x=80pt]
  \node (1) at (0,0.5) {$\mathsf{H}$};
  \node (2) at (1,0.5) {$\infty\mathsf{Gpd}$};
  \path[semithick,->]
    ([yshift=5pt]1.east) edge node[descr,pos=.5,scale=0.75] {$\Gamma$} ([yshift=5pt]2.west)
    ([yshift=-15pt]1.east) edge node[descr,pos=.5,scale=0.75] {$\Pi_\infty$} ([yshift=-15pt]2.west);
  \path[semithick,{Hooks[left]}->]
    ([yshift=15pt]2.west) edge ([yshift=15pt]1.east)
    ([yshift=-5pt]2.west) edge node[descr,pos=.5,scale=0.75] {$\mathrm{const}$} ([yshift=-5pt]1.east);
\end{tikzpicture}
\end{gathered}
\end{equation}
of adjoint $\infty$-functors in which both adjoints to $\Gamma$ are fully faithful $\infty$-functors, and $\Pi_\infty$ preserves products. In similar form, this structure already appears in \cite{ST_deRhamInftyStacks} and \cite{MR1731635}.

We have come close to exhibiting our $\infty$-topoi $\mathsf{St}(\mathcal{C}, \tau)$ as cohesive $\infty$-topoi. Indeed, we are only missing that the pro-left adjoint to $\mathrm{const}$ should actually land in $\infty\mathsf{Gpd}$ rather than in $\mathrm{Pro}(\infty\mathsf{Gpd})$, and that it preserves finite products. It seems altogether possible that this is true in the analytic category, for \cite[\S 2.16]{ST_deRhamInftyStacks} constructs this functor explicitly at the level of the homotopy category. It is doubtful, though, that the same holds in the algebraic context, since there $\Pi_\infty$ should encode, among other things, the {\' e}tale fundamental group, which is often a true pro-algebraic group.

If $j_!$ in \eqref{eq:QuadrupleFromReductionFunctor} also lifts to the $\infty$-level, \eqref{eq:InfinitesimalCohesion} realizes what Schreiber calls an \emph{infinitesimal cohesive neighborhood} \cite[\S 3.5]{arXiv1310.7930}. Although we do believe that this should be the case for any and all of our choices of $(\mathcal{C}, \tau)$, we do not have a proof of this fact.

\subsection{The case of a smooth projective variety}
\label{section:CaseOfSmoothProjectiveVariety}

In this section we restrict ourselves to the analytic topology. The algebraic counterpart is the concern of the next section.

\subsubsection{The classical nonabelian Hodge correspondence.}
\label{section:ClassicalNonabelianHodgeCorrespondence}

For $X$ a smooth projective variety, the $(n+1)$-simplices of the {\v C}ech nerve of $X \to X_\mathrm{dR}$ are given by the formal completion of the main diagonal in $X^{\times n}$. Furthermore, the $\infty$-colimit of this simplicial object agrees with the ordinary colimit of its 1-truncation, and so $X_\mathrm{dR}$ can be simply realized as the quotient of $X$ by the formal completion of the diagonal in $X \times X$:
\[
X_\mathrm{dR} = \Big[ (X \times X)_\Delta^\wedge \rightrightarrows X \Big]
\]
Coherent sheaves over $X_\mathrm{dR}$ are then easily seen to be the same thing as crystals of coherent sheaves on $X$ in the sense of Grothendieck \cite{MR0269663}: that is, vector bundles with flat connection.

On the other hand, Higgs bundles can be codified as vector bundles on the so-called \emph{Dolbeault stack} of $X$. The latter is the deformation of $X_\mathrm{dR}$ to the normal cone, and can be realized as the quotient of $X$ by the formal completion of the zero section in the total space of its tangent bundle \cite{MR1492538}:
\[
X_\mathrm{Dol} := \Big[ TX_{\,0}^\wedge \rightrightarrows X \Big]
\]
The nonabelian Hodge theorem can be succintly expressed in terms of these stacks.

\begin{theorem}[\cite{MR1179076}] \label{th:NonabelianHodgeTheoremGTorsors}
For $G$ a linear algebraic group over $\CC$, there is an equivalence
\[
\Map\!\big( X_\mathrm{dR}, BG \big) \simeq \Map\!\big( X_\mathrm{Dol}, BG \big)^{\mathrm{ss}, 0}
\]
\end{theorem}

\subsubsection{}
\label{section:GerbesOverXdRAndXDol}

Let $A$ be an abelian linear algebraic group over $\CC$. The cohomology of $X_\mathrm{dR}$ (resp., $X_\mathrm{Dol}$) with coefficients in $A$ can be expressed in terms of the de Rham (resp., Dolbeault) complex of $X$ with coefficients in $A$ introduced in Definition \ref{definition:FlatAGerbeCech} (resp., Definition \ref{definition:HiggsAGerbeCech}):
\[
\mathrm{dR}^{A}_X := \Big[ A \xrightarrow{\;\;a \,\mapsto\, a^{-1}da\;\;} \Omega^1_X \otimes \mathfrak{a} \xrightarrow{\quad d \quad} \Omega^2_X \otimes \mathfrak{a} \xrightarrow{\quad d \quad} \cdots \Big]
\]
\[
\left( \text{resp., }\quad \mathrm{Dol}^{A}_X := \Big[ A \xrightarrow{\quad 0 \quad} \Omega^1_X \otimes \mathfrak{a} \xrightarrow{\quad 0 \quad} \Omega^2_X \otimes \mathfrak{a} \xrightarrow{\quad 0 \quad} \cdots \Big] \quad\right).
\]

\vspace{.5em}
\begin{lemma}
$\pi_i \Map\!\big( X_\mathrm{dR}, B^k A \big) \cong \mathbb{H}^{k-i} \big( X, \mathrm{dR}^{A}_X \big)$.
\end{lemma}

\begin{proof}
The holomorphic Poincar{\' e} lemma asserts that $\mathrm{dR}^{A}_X$ is a resolution of the sheaf $\mathrm{dis}(A)$. Using \eqref{eq:deRhamGeometricMorphism} and Proposition \ref{proposition:Delta}, we have
\begin{align*}
\pi_i \Map\!\big( X_\mathrm{dR}, B^k A \big) &\cong \pi_i \Map\!\big( X, \delta (B^k A) \big) \cong \pi_i \Map\!\big( X, B^k(\mathrm{dis}(A)) \big) \\ &= H^{k-i} \big( X, \mathrm{dis}(A) \big) \cong \mathbb{H}^{k-i} \big( X, \mathrm{dR}^{A}_X \big). \tag*{\qedhere}
\end{align*}
\end{proof}

\begin{lemma} \label{lemma:CohomologyDolbeaultStack}
$\pi_i \Map\!\big( X_\mathrm{Dol}, B^k A \big) \cong \mathbb{H}^{k-i} \big( X, \mathrm{Dol}^{A}_X \big)$.
\end{lemma}

\begin{proof}
An abelian linear algebraic group factors as a direct sum of copies of $\GG_m$, $\GG_a$ and a finite abelian group (see \S\ref{section:LinearAlgebraicGroups}). Since delooping commutes with finite products, it is enough to check each of these cases separately. For $A = \GG_a$ the calculation is classical (see, e.g., \cite[Proposition 3.1]{arXiv_alggeom_9712020}) and we omit the proof.

Let $F$ be a discrete group ---a class to which finite groups belong---, so that $F \simeq \mathrm{dis}(F)$. Since $\mathrm{dis} = \mathrm{const} \circ \Gamma$ and $\mathrm{const}$ has $\Pi_\infty$ as a pro-left adjoint, we have
\begin{align*}
\Map\!\big( \mathfrak{X} , B^k F \big) \simeq \Map\!\big( \mathfrak{X} , \mathrm{const}\, B^k (\Gamma F) \big) \simeq  \Map\!\big( \Pi_\infty \mathfrak{X} , B^k (\Gamma F) \big);
\end{align*}
that is, the cohomology of a stack $\mathfrak{X}$ with coefficients in $F$ depends only on its fundamental pro-$\infty$-groupoid. We now claim that
\begin{equation} \label{eq:FundamentalProInftyGroupoidDolbeaultStack}
\Pi_\infty X_\mathrm{Dol} \simeq \Pi_\infty X,
\end{equation}
which implies the lemma for $A = F$.

In order to prove \eqref{eq:FundamentalProInftyGroupoidDolbeaultStack}, we proceed by a series of reductions\footnote{Pun definitely intended.}. For the first one we observe that $X_\mathrm{Dol}$ is defined as the quotient of $X$ by $TX_{\,0}^\wedge$, and that $\Pi_\infty$ commutes with $\infty$-colimits. Hence it is enough to check that
\[
\Pi_\infty \big( TX_{\,0}^\wedge \big)^{\times_X m} \simeq \Pi_\infty X.
\]
for every $m \geq 1$. But $TX_{\,0}^\wedge$ is given as the colimit of the infinitesimal neighborhoods $TX_{\,\,0}^{[s]}$ of the zero section in $TX$, and so we only need to show that
\[
\Pi_\infty \big( TX_{\,\,0}^{[s]} \big)^{\times_X m} \simeq \Pi_\infty X.
\]
for every $m \geq 1$ and $s \geq 1$. Now, since both $X$ and $\big( TX_{\,\,0}^{[s]} \big)^{\times_X m}$ are representable spaces (in the terminology of the last section), and $\mathrm{Red} \big( TX_{\,\,0}^{[s]} \big)^{\times_X m} \cong X$, we have
\[
j^\ast \left( \big( TX_{\,\,0}^{[s]} \big)^{\times_X m} \right) = \mathrm{red} \left( \big( TX_{\,\,0}^{[s]} \big)^{\times_X m} \right) \cong \mathrm{red} \circ j \circ \mathrm{red} \left( \big( TX_{\,\,0}^{[s]} \big)^{\times_X m} \right) \cong \mathrm{red}\!\left( X \right) = j^\ast\!\left( X \right)
\]
and \eqref{eq:FundamentalProInftyGroupoidAndReduction} finishes the proof:
\[
\Pi_\infty \big( TX_{\,\,0}^{[s]} \big)^{\times_X m} \simeq\Pi_{\infty\mathrm{red}} \circ j^\ast \left( \big( TX_{\,\,0}^{[s]} \big)^{\times_X m} \right) \simeq \Pi_{\infty\mathrm{red}} \circ j^\ast\!\left( X \right) \simeq \Pi_\infty X.
\]

Finally, the exponential sequence and the cases of $\GG_a$ and $\ZZ$ ---a discrete group for sure--- imply the statement for $A = \GG_m$.
\end{proof}

The above computations suggest what our coveted cover- and cocycle-independent versions of Definitions \ref{definition:FlatAGerbeCech} and \ref{definition:HiggsAGerbeCech} should be.

\begin{definition} \label{definition:FlatGerbe}
A flat $A$-gerbe over $X$ is an $A$-gerbe over $X_\mathrm{dR}$.
\end{definition}

\begin{definition} \label{definition:HiggsGerbe}
A Higgs $A$-gerbe over $X$ is an $A$-gerbe over $X_\mathrm{Dol}$.
\end{definition}

\subsubsection{The Hodge correspondence for gerbes.}
\label{section:HodgeCorrespondenceGerbes}

Since the coefficient groups of gerbes are abelian, we might expect that abelian Hodge theory should relate their flat and Higgs versions. There is, however, one important restriction ---which is, in fact, the only thing that makes the twisted correspondence we aim to prove non-trivial. Namely, that $A$ cannot contain any algebraic torus. Indeed, there is no hope of establishing an equivalence between $\Map\!\big(X_\mathrm{dR}, B^2\GG_m\big)$ and $\Map\!\big(X_\mathrm{Dol}, B^2\GG_m\big)$ or a full subcategory thereof, for the automorphism 1-categories of objects in these (2-)categories are given by
\[
\Omega\Map\!\big(X_\mathrm{dR}, B^2\GG_m\big) \simeq \Map\!\big(X_\mathrm{dR}, B\GG_m\big)
\]
and
\[
\Omega\Map\!\big(X_\mathrm{Dol}, B^2\GG_m\big) \simeq \Map\!\big(X_\mathrm{Dol}, B\GG_m\big),
\]
respectively. But Theorem \ref{th:NonabelianHodgeTheoremGTorsors} establishes an equivalence between the first of these and the full subcategory of the second on the degree zero Higgs $\GG_m$-torsors  (the semistability condition is trivial in this case, as is the vanishing of the second rational Chern class: see below). This shows that any relation between $\Map\!\big(X_\mathrm{dR}, B^2\GG_m\big)$ and $\Map\!\big(X_\mathrm{Dol}, B^2\GG_m\big)$ would have to involve restrictions not only on objects, but also on 1-morphisms.

\begin{proposition} \label{proposition:HodgeCorrespondenceGerbes}
Suppose $A \cong \GG_a^{\oplus m} \oplus F$, where $F$ is a finite group. Then,
\[
\Map\!\big( X_\mathrm{dR}, B^2 A \big) \simeq \Map\!\big( X_\mathrm{Dol}, B^2 A \big).
\]
\end{proposition}

\begin{proof}
As in the proof of Lemma \ref{lemma:CohomologyDolbeaultStack}, it is enough to show the statement independently for the cases $A = \GG_a$ and $A = \mu_n$. The first follows from abelian Hodge theory, while the second holds true because $\mu_n$ has trivial Lie algebra, and hence
\begin{equation*}
\Map\!\big( X_\mathrm{dR}, B^2\mu_n \big) \simeq \Map\!\big( X, B^2\mu_n \big) \simeq \Map\!\big( X_\mathrm{Dol}, B^2\mu_n \big). \tag*{\qedhere}
\end{equation*}
\end{proof}

\subsubsection{}
\label{section:SemistabilityForAbelianTorsors}

Although Proposition \ref{proposition:HodgeCorrespondenceGerbes} concerns gerbes, it also contains within itself a statement about the category of $A$-torsors on $X_\mathrm{Dol}$, since the latter appears as the automorphism 1-category of the distinguished object of $\Map\!\big( X_\mathrm{Dol}, B^2 A \big)$. Namely, that the conditions on semistability and vanishing of Chern classes of Theorem \ref{th:NonabelianHodgeTheoremGTorsors} are always satisfied. As we remarked above, this is not quite so for $\GG_m$-torsors on $X_\mathrm{Dol}$: we do need to impose that they are of degree zero ---semistability follows from it.

Indeed, a $\GG_m$-torsor on $X_\mathrm{Dol}$, $\mathcal{L}$, sends an irreducible representation of $\GG_m$ ---that is, a character--- to a Higgs line bundle, for which the semistability condition is obviously vacuous, as is the vanishing of the second Chern class. If we require that the first rational Chern class of these line bundles coming from irreps also vanishes, then $\mathcal{L}$ sends an arbitrary representation of $\GG_m$ to a direct sum of degree zero Higgs line bundles, which is certainly semistable.

We can also prove the claim for $A = \mu_n$ and $A = \GG_a$ without invoking the Hodge correspondence for gerbes. In the first case we can use the same argument of the last paragraph, complemented with the observation that the Higgs line bundles we obtain now are torsion, and so their first rational Chern class vanishes automatically.

For $A=\GG_a$ the resulting vector bundles are unipotent, and the Higgs fields on them, nilpotent. These Higgs bundles can hence be written as successive extensions of the trivial line bundle equipped with the zero Higgs field. We finish by observing that the category of semistable Higgs bundles with vanishing first and second rational Chern classes is closed under extensions.

\subsection{Analytification}
\label{section:Analytification}

The main purpose of this section is to show that the Hodge correspondences ---Theorem \ref{th:NonabelianHodgeTheoremGTorsors} and Proposition \ref{proposition:HodgeCorrespondenceGerbes}--- also hold in the algebraic setting under appropriate assumptions. Our results here will allow us to work in \S\ref{section:Proof} without explicit reference to the {\' e}tale or analytic topologies (see \S\ref{section:StatementMainTheorems}).

\subsubsection{Finiteness conditions on the algebraic side.}

Let $a: \mathsf{Aff}_{\CC, \text{ft}} \hookrightarrow \mathsf{Aff}_\CC$ denote the inclusion functor of the category of complex affine schemes of finite type into that of all complex affine schemes. With respect to the {\' e}tale topologies on both sides, $a$ is both continuous and cocontinuous and preserves finite limits, and hence the triple of adjoint functors $a_! \dashv a^\ast \dashv a_\ast$ between categories of presheaves of sets descends to their respective categories of sheaves of sets, and ultimately lifts to their 1-localic $\infty$-topoi of $\infty$-stacks via Proposition \ref{proposition:LiftingGeometricMorphisms}:
\[
\begin{tikzpicture}[x=120pt]
  \node (1) at (0,0.5) {$\mathsf{St}(\mathsf{Aff}_{\CC, \mathrm{ft}}, \text{{\' e}t})$};
  \node (2) at (1,0.5) {$\mathsf{St}(\mathsf{Aff}_\CC, \text{{\' e}t})$};
  \path[semithick,->]
    (2.west) edge node[descr,pos=.5,scale=0.75] {$a^\ast$} (1.east);
  \path[semithick,{Hooks[right]}->]
    ([yshift=-10pt]1.east) edge node[descr,pos=.5,scale=0.75] {$a_\ast$} ([yshift=-10pt]2.west)
    ([yshift=10pt]1.east) edge node[descr,pos=.5,scale=0.75] {$a_!$} ([yshift=10pt]2.west);
\end{tikzpicture}
\]
An object $\mathfrak{X} \in \mathsf{St}(\mathsf{Aff}_\CC, \text{{\' e}t})$ is said to be \emph{almost locally of finite type} \cite{NotesGL_Stacks} if it belongs to the essential image of $a_!$. Because $a_!$ preserves $\infty$-colimits and finite $\infty$-limits, it is easy to see from their constructions that the de Rham and Dolbeault stacks of a smooth projective variety do belong in this subcategory, as is the case for linear algebraic groups over $\CC$.

\subsubsection{The analytification functor.}

Given a complex affine scheme \emph{of finite type}, we can produce a complex analytic space {\` a} la Serre \cite{MR0082175} in a functorial manner. This analytification functor, $\mathrm{an}: \mathsf{Aff}_{\CC, \mathrm{ft}} \to \mathsf{An}_\CC$, is continuous with respect to the {\' e}tale topologies on both sides and preserves finite limits; the same procedure we have used several times now produces an adjoint pair
\[
\begin{tikzpicture}[x=100pt]
  \node (1) at (0,0.5) {$\mathsf{St}(\mathsf{Aff}_{\CC, \mathrm{ft}}, \text{{\' e}t})$};
  \node (2) at (1,0.5) {$\mathsf{St}(\mathsf{An}_\CC, \text{{\' e}t})$};
  \path[semithick,->]
    ([yshift=-5pt]2.west) edge node[descr,pos=.5,scale=0.75] {$\mathrm{an}^\ast$} ([yshift=-5pt]1.east)
    ([yshift=5pt]1.east) edge node[descr,pos=.5,scale=0.75] {$\mathrm{an}$} ([yshift=5pt]2.west);
\end{tikzpicture}
\]
The left adjoint, which we have denoted $\mathrm{an}$ abusing terminology, is induced by $\mathrm{an}_!$ and extends the original analytification functor to all those $\infty$-stacks on the {\' e}tale site that are almost locally of finite type.

Given a smooth complex projective variety $X$, denote by $X_\mathrm{dR}$ (resp. $X_\mathrm{Dol}$) its de Rham (resp. Dolbeault) stack as constructed in the {\' e}tale topology following the recipe of \S\ref{section:ClassicalNonabelianHodgeCorrespondence}. Considering $X$ as a complex-analytic space through the analytification functor above, we can also look at its de Rham (resp,. Dolbeault) stack, this time constructed in the analytic sense; we temporarily denote the latter by $X_\mathrm{dR}^\mathrm{an}$ (resp., $X_\mathrm{Dol}^\mathrm{an}$). Because analytification commutes with $\infty$-colimits and finite $\infty$-limits, we have
\[
\mathrm{an}(X_\mathrm{dR}) \simeq X_\mathrm{dR}^\mathrm{an} \qquad \Big(\text{resp., }\mathrm{an}(X_\mathrm{Dol}) \simeq X_\mathrm{Dol}^\mathrm{an}\Big)
\]
Moreover, if $G$ is a linear algebraic group over $\CC$ (resp. $A$ is an abelian linear algebraic group over $\CC$), considered as an {\' e}tale scheme, denote by $G^\mathrm{an}$ (resp. $A^\mathrm{an}$) its analytification.

The following proposition expresses the fact that the categories we are interested in are the same in the algebraic and analytic cases, so that Theorem \ref{th:NonabelianHodgeTheoremGTorsors} and Proposition \ref{proposition:HodgeCorrespondenceGerbes} also hold in the {\' e}tale topology.

\begin{proposition}
\label{proposition:AlgebraicHodgeCorrespondence}
For $G$ a reductive linear algebraic group over $\CC$, and $A \cong \GG_a^{\oplus m} \oplus F$, where $F$ is a finite group, we have
\begin{enumerate}[topsep=-3pt,partopsep=0pt,parsep=2pt,itemsep=2pt]
\item $\Map\!\big( X_\mathrm{dR}, BG \big) \simeq \Map\!\big( X_\mathrm{dR}^{an}, BG^{an} \big)$
\item $\Map\!\big( X_\mathrm{Dol}, BG \big) \simeq \Map\!\big( X_\mathrm{Dol}^{an}, BG^{an} \big)$
\item $\Map\!\big( X_\mathrm{dR}, B^k A \big) \simeq \Map\!\big( X_\mathrm{dR}^{an}, B^k A^{an} \big)$
\item $\Map\!\big( X_\mathrm{Dol}, B^k A \big) \simeq \Map\!\big( X_\mathrm{Dol}^{an}, B^k A^{an} \big)$
\end{enumerate}
\end{proposition}

\begin{proof}
For a proof of (i), we refer the reader to \cite[Theorem 9.2]{MR1397992}. Statement (ii) follows from Serre's GAGA \cite{MR0082175}, which implies that analytic Higgs vector bundles are in fact algebraic.

For the case $A = \GG_a$, (iii) follows from the usual comparison theorem of de Rham cohomology in the algebraic and analytic settings \cite{MR0269663}, while (iv) is once again a direct consequence of GAGA. The comparison theorem between {\' e}tale and analytic cohomology \cite[Theorem 3.12]{MR559531} proves both (iii) and (iv) in case $A$ is a finite abelian group.
\end{proof}

\section{The twisted correspondence}
\label{section:TwistedCorrespondence}

\subsection{Torsion phenomena in the vector bundle case}
\label{section:TorsionPhenomena}

With all the technical baggage of the last two sections under our belt, we return here to the case of vector bundles ---or, rather, $GL_n$-torsors. Our discussion in this section will serve to illustrate the chief difficulty in the obvious approach to proving a twisted nonabelian Hodge correspondence and point out the main idea of the workaround.

\subsubsection{}

Recall that in \S\ref{section:Introduction} we defined twisted vector bundles, twisted connections and twisted Higgs fields in terms of a cover $\mathfrak{U}$ of our smooth complex projective variety $X$. We condensed all of these concepts into a few definitions, that we detail in the following table.

\vspace{1.5em}
\begin{center}
\begin{tabular}{cc}
\toprule
de Rham side & Dolbeault side \\
\midrule
\begin{tabular}{c}
Flat $\mathfrak{U}$-$\GG_m$-gerbe \\ (Definition \ref{definition:FlatGmGerbeCech})
\end{tabular}
&
\begin{tabular}{c}
Higgs $\mathfrak{U}$-$\GG_m$-gerbe \\ (Definition \ref{definition:HiggsGmGerbeCech})
\end{tabular}
\\
\begin{minipage}{160pt}
\mbox{}
\end{minipage}
&
\begin{minipage}{160pt}
\mbox{}
\end{minipage}
\\
\begin{tabular}{c}
Basic vector bundle on \\ a flat $\mathfrak{U}$-$\GG_m$-gerbe \\ (Definition \ref{definition:BasicVectorBundleOverFlatGmGerbeCech})
\end{tabular}
&
\begin{tabular}{c}
Basic vector bundle on \\ a Higgs $\mathfrak{U}$-$\GG_m$-gerbe \\ (Definition \ref{definition:BasicVectorBundleOverHiggsGmGerbeCech})
\end{tabular}
\\
\bottomrule
\end{tabular}
\end{center}
\vspace{1.5em}

Our work in the \S\ref{section:GeometrizingTwistedTorsors} and \S\ref{section:HodgeTheory} was aimed at giving cover- and cocycle-independent versions of all of these definitions. We remarked at the end of \S\ref{section:GerbesOverXdRAndXDol} that we had succeeded in doing so for the objects in the first row of the table above. The appropriate counterparts to the definitions in the second row are simply given by Definition \ref{definition:BasicHTorsorOverGerbe} applied to  Definitions \ref{definition:FlatGerbe} and \ref{definition:HiggsGerbe}.

\vspace{1.5em}
\begin{center}
\begin{tabular}{cc}
\toprule
de Rham side & Dolbeault side \\
\midrule
\begin{tabular}{c}
Flat $\GG_m$-gerbe \\ (Definition \ref{definition:FlatGerbe}, $A = \GG_m$)
\end{tabular}
&
\begin{tabular}{c}
Higgs $\GG_m$-gerbe \\ (Definition \ref{definition:HiggsGerbe}, $A = \GG_m$)
\end{tabular}
\\
\begin{minipage}{160pt}
\mbox{}
\end{minipage}
&
\begin{minipage}{160pt}
\mbox{}
\end{minipage}
\\
\begin{tabular}{c}
Basic $GL_n$-torsor on \\ a flat $\GG_m$-gerbe \\ (Definition \ref{definition:BasicHTorsorOverGerbe}, $A = \GG_m$, $H = GL_n$)
\end{tabular}
&
\begin{tabular}{c}
Basic $GL_n$-torsor on \\ a Higgs $\GG_m$-gerbe \\ (Definition \ref{definition:BasicHTorsorOverGerbe}, $A = \GG_m$, $H = GL_n$)
\end{tabular}
\\
\bottomrule
\end{tabular}
\end{center}
\vspace{1.5em}

The relation between these two tables is that between {\v C}ech and sheaf cohomology. 
Indeed, suppose $\big( \underline{\alpha}, \underline{\omega}, \underline{F} \big) \in \check{\ZZ}^2 \big( \mathfrak{U}, \mathrm{dR}^{\GG_m}_X \big)$ is a flat $\mathfrak{U}$-$\GG_m$-gerbe, and let $\theta \in \Map\!\big( X_\mathrm{dR}, B^2 \GG_m \big)_0$ be a flat $\GG_m$-gerbe. Then, if the obvious compatibility condition between these two pieces of data ---namely, that the image of $\left[\big( \underline{\alpha}, \underline{\omega}, \underline{F} \big)\right] \in \check{\HH}^2 \big( \mathfrak{U}, \mathrm{dR}^{\GG_m}_X \big)$ in sheaf cohomology coincides with $\left[ \theta \right] \in \HH^2 \big( X, \mathrm{dR}^{\GG_m}_X \big)$---, the category of basic vector bundles on $\big( \underline{\alpha}, \underline{\omega}, \underline{F} \big)$ is equivalent to that of basic $GL_n$-torsors on ${}_\theta (X_\mathrm{dR})$. Of course, the parallel statement about basic vector bundles on Higgs $\mathfrak{U}$-$\GG_m$-gerbes and basic $GL_n$-torsors on Higgs $\GG_m$-gerbes also holds.

Theorem \ref{theorem:MainTheoremGLnCech} follows then from what will be the final form of our theorem for the case of vector bundles.

\begin{theorem}
\label{theorem:MainTheoremGLn}
Let $\theta \in \Map\!\big(X_\mathrm{dR}, B^2\GG_m\big)_0$ be a flat $\GG_m$-gerbe over $X$. Then there is a Higgs $\GG_m$-gerbe over X, $\widetilde\theta  \in \Map\!\big(X_\mathrm{Dol}, B^2\GG_m\big)_0$, for which there is an equivalence
\[
\Map_{B\GG_m} \!\big( {}_\theta(X_\mathrm{dR}), BGL_n \big) \simeq \Map_{B\GG_m} \!\big( {}_{\widetilde\theta}(X_\mathrm{Dol}), BGL_n \big)^\mathrm{ss}
\]
Conversely, given $\widetilde\theta  \in \!\Map\big(X_\mathrm{Dol}, B^2\GG_m\big)_0$ we can find $\theta \in \!\Map\big(X_\mathrm{dR}, B^2\GG_m\big)_0$ such that the same conclusion holds.
\end{theorem}

The semistability conditions that define the right hand side of this correspondence as a full subcategory of the category of basic $GL_n$-torsors on ${}_{\widetilde\theta}(X_\mathrm{Dol})$ are rather difficult to state at this point. We will build the requisite language in \S\ref{section:ToruslessGerbes}, and elucidate them in \S\ref{section:LiftingFlatHiggsAGerbes} (Definitions \ref{definition:SemistableBasicHPrimeTorsor} and \ref{definition:SemistableRectifiableBasicHTorsor}).

\subsubsection{}
\label{section:ObviousAttempt}

Let $\theta \in \Map\!\big( X_\mathrm{dR}, B^2\GG_m \big)_0$ be a flat $\GG_m$-gerbe over $X$, and $\widetilde\theta \in \Map\!\big( X_\mathrm{Dol}, B^2\GG_m \big)_0$ a Higgs $\GG_m$-gerbe over the same variety. According to \eqref{eq:BasicHTorsors}, we can express the categories of basic $GL_n$-torsors on one and the other as limits:
\[
\Map_{B\GG_m} \!\big( {}_{\theta}(X_\mathrm{dR}), BGL_n \big) \simeq \lim \left\{
\begin{tikzpicture}[x=135pt,y=35pt,baseline={([yshift=-5pt]current bounding box.west)}]
  \node (12) at (1,1) {$\ast$};
  \node (21) at (0,0) {$\Map\!\big(X_\mathrm{dR}, B\PP GL_n\big)$};
  \node (22) at (1,0) {$\Map\!\big(X_\mathrm{dR}, B^2 \GG_m\big)$};
  \path[semithick,->]
    (12) edge[shorten <=-2pt] node[pos=0.4,scale=0.75,xshift=6pt] {$\theta$} (22)
    (21) edge node[pos=0.5,scale=0.75,yshift=8pt] {$ob_{\GG_m}$} (22);
\end{tikzpicture}
\right\}
\]
\[
\Map_{B\GG_m} \!\big( {}_{\widetilde\theta}(X_\mathrm{Dol}), BGL_n \big) \simeq \lim \left\{
\begin{tikzpicture}[x=135pt,y=35pt,baseline={([yshift=-5pt]current bounding box.west)}]
  \node (12) at (1,1) {$\ast$};
  \node (21) at (0,0) {$\Map\!\big(X_\mathrm{Dol}, B\PP GL_n\big)$};
  \node (22) at (1,0) {$\Map\!\big(X_\mathrm{Dol}, B^2 \GG_m\big)$};
  \path[semithick,->]
    (12) edge[shorten <=-2pt] node[pos=0.4,scale=0.75,xshift=6pt] {$\widetilde\theta$} (22)
    (21) edge node[pos=0.5,scale=0.75,yshift=8pt] {$ob_{\GG_m}$} (22);
\end{tikzpicture}
\right\}
\]
Notice how, as we mentioned several times in the introduction, a basic $GL_n$-torsor on ${}_{\theta}(X_\mathrm{dR})$ (resp., ${}_{\widetilde\theta}(X_\mathrm{Dol})$) determines an honest, untwisted $\PP GL_n$-torsor on $X_\mathrm{dR}$ (resp., $X_\mathrm{Dol}$).

The natural attempt at proving Theorem \ref{theorem:MainTheoremGLn} would be to try to relate the terms in one of these limits to the matching ones in the other in a functorial manner. One of the comparisons is easy: the classical nonabelian Hodge correspondence (Theorem \ref{th:NonabelianHodgeTheoremGTorsors}) provides an equivalence 
\[
\Map\!\big(X_\mathrm{dR}, B\PP GL_n\big) \simeq \Map\!\big(X_\mathrm{Dol}, B\PP GL_n\big)^{\mathrm{ss}, 0}
\]
between the category of flat $\PP GL_n$-torsors on $X$ and the full subcategory of the category of Higgs $\PP GL_n$-torsors on $X$ on the semistable objects with vanishing Chern numbers ---this is, in fact, one of the two stability conditions we will need to impose on the Dolbeault side of our correspondence. However, as we saw in \S\ref{section:HodgeCorrespondenceGerbes}, the Hodge correspondence fails for $\GG_m$-gerbes.

Observe, though, that the above presentation of $\Map_{B\GG_m} \!\big( {}_{\widetilde\theta}(X_\mathrm{Dol}), BGL_n \big)$ as a limit implies that automorphisms of $\Map\!\big(X_\mathrm{Dol}, B^2\GG_m\big)$ enter into it at the level of objects: restricting these would then be a restriction on objects of the category of basic $GL_n$-torsors on ${}_{\widetilde\theta}(X_\mathrm{Dol})$, and so the hope for the existence of a twisted nonabelian Hodge correspondence is not all lost.

\subsubsection{}

It is the torsion phenomena that we referred to in the introduction that allows us to bypass this difficulty; they all stem from the fact that the determinant map $\det : GL_n \to \GG_m$ is a surjective group homomorphism that remains surjective when restricted to its center. These surjectivity properties allow us to write the following commutative diagram of linear algebraic groups and homomorphisms with exact rows and columns:
\begin{equation} \label{eq:MasterDiagramGL_n}
\begin{gathered}
\begin{tikzpicture}[x=200pt, y=170pt]
  \node (12) at (0.2,1) {$1$};
  \node (13) at (0.4,1) {$1$};
  \node (21) at (0,0.8) {$1$};
  \node (22) at (0.2,0.8) {$\mu_n$};
  \node (23) at (0.4,0.8) {$SL_n$};
  \node (24) at (0.6,0.8) {$\PP GL_n$};
  \node (25) at (0.8,0.8) {$1$};
  \node (31) at (0,0.6) {$1$};
  \node (32) at (0.2,0.6) {$\GG_m$};
  \node (33) at (0.4,0.6) {$GL_n$};
  \node (34) at (0.6,0.6) {$\PP GL_n$};
  \node (35) at (0.8,0.6) {$1$};
  \node (42) at (0.2,0.4) {$\GG_m$};
  \node (43) at (0.4,0.4) {$\GG_m$};
  \node (52) at (0.2,0.2) {$1$};
  \node (53) at (0.4,0.2) {$1$};
  \draw[double distance=1.5pt] (24) -- (34);
  \draw[double distance=1.5pt] (42) -- (43);
  \path[semithick,->]
    (12) edge (22)
    (13) edge (23)
    (21) edge (22)
    (22) edge (23)
    (22) edge (32)
    (23) edge (24)
    (23) edge (33)
    (24) edge (25)
    (31) edge (32)
    (32) edge (33)
    (32) edge node[pos=0.45,scale=0.8,xshift=15pt] {$(-)^n$} (42)
    (33) edge (34)
    (33) edge node[pos=0.45,scale=0.8,xshift=11pt] {$\det$} (43)
    (34) edge (35)
    (42) edge (52)
    (43) edge (53);
\end{tikzpicture}
\end{gathered}
\end{equation}
From the Puppe sequences associated to the exact sequences above, we deduce a diagram of $\infty$-stacks, in which the rows and the column are fiber sequences:
\begin{equation} \label{eq:TorsionDiagramGL_n}
\begin{gathered}
\begin{tikzpicture}[x=160pt, y=75pt]
  \node (11) at (0,1) {$BSL_n$};
  \node (12) at (0.5,1) {$B\PP GL_n$};
  \node (13) at (1,1) {$B^2\mu_n$};
  \node (21) at (0,0.5) {$BGL_n$};
  \node (22) at (0.5,0.5) {$B\PP GL_n$};
  \node (23) at (1,0.5) {$B^2\GG_m$};
  \node (33) at (1,0) {$B^2\GG_m$};
  \draw[double distance=1.5pt] (12) -- (22);
  \path[semithick,->]
    (11) edge (12)
    (11) edge (21)
    (12) edge node[pos=0.5,scale=0.8,yshift=10pt] {$ob_{\mu_n}$} (13)
    (13) edge node[pos=0.45,scale=0.8,xshift=17pt] {$ob_{\GG_m}^{\mu_n}$} (23)
    (21) edge (22)
    (22) edge node[pos=0.5,scale=0.8,yshift=10pt] {$ob_{\GG_m}$} (23)
    (23) edge node[pos=0.45,scale=0.8,xshift=15pt] {$(-)^n$} (33);
\end{tikzpicture}
\end{gathered}
\end{equation}
The second named horizontal map is the universal obstruction $\GG_m$-gerbe for $\PP GL_n$-torsors, in the sense that, for every $\infty$-stack $\mathfrak{X}$, it induces the map that associates to a $\PP GL_n$-torsor on $\mathfrak{X}$ its obstruction $\GG_m$-gerbe (see \S\ref{section:ObstructionGerbes}). The fact that this map factors through $B^2\mu_n$ shows that the $n$th-power of the obstruction $\GG_m$-gerbe of a $\PP GL_n$-torsor is always trivializable.

\begin{lemma} \label{lemma:NoBasicGLnTorsorsForNonTorsionGerbes}
Let $\alpha \in \Map\!\big( \mathfrak{X}, B^2 \GG_m \big)_0$ be a $\GG_m$-gerbe over an $\infty$-stack $\mathfrak{X}$. Then, the category of basic $GL_n$-torsors on ${}_\alpha \mathfrak{X}$ is empty unless $\alpha^n$ is a trivializable $\GG_m$-gerbe.
\end{lemma}

\begin{proof}
If $\alpha^n$ is not trivializable, the image of the two morphisms into $\Map\!\big( \mathfrak{X}, B^2 \GG_m \big)$ in
\[
\Map_{B\GG_m} \!\big( {}_\alpha \mathfrak{X}, BGL_n \big) \simeq \lim \left\{
\begin{tikzpicture}[x=125pt,y=35pt,baseline={([yshift=-5pt]current bounding box.west)}]
  \node (12) at (1,1) {$\ast$};
  \node (21) at (0,0) {$\Map\!\big(\mathfrak{X}, B\PP GL_n\big)$};
  \node (22) at (1,0) {$\Map\!\big(\mathfrak{X}, B^2 \GG_m\big)$};
  \path[semithick,->]
    (12) edge[shorten <=-2pt] node[pos=0.4,scale=0.75,xshift=8pt] {$\alpha$} (22)
    (21) edge node[pos=0.5,scale=0.75,yshift=8pt] {$ob_{\GG_m}$} (22);
\end{tikzpicture}
\right\}
\]
land in different connected components.
\end{proof}

This provides a uniform explanation for all the occurrences of torsion in the introduction, from that of the bare, underlying $\GG_m$-gerbe to that of the twisted connections and twisted Higgs fields. It also suggests how we might get around the fact that the Hodge correspondence does not hold for $\GG_m$-gerbes: by using that it does for $\mu_n$-gerbes. Indeed, lifting $\theta \in \Map\!\big(X_\mathrm{dR}, B^2\GG_m\big)_0$ (resp., $\widetilde\theta \in \!\Map\big(X_\mathrm{Dol}, B^2\GG_m\big)_0$) to a $\mu_n$-gerbe $\theta' \in \Map\!\big(X_\mathrm{dR}, B^2\mu_n\big)_0$ (resp., $\widetilde\theta' \in \!\Map\big(X_\mathrm{Dol}, B^2\mu_m\big)_0$), the approach of \S\ref{section:ObviousAttempt} establishes an equivalence between
\[
\lim \left\{
\begin{tikzpicture}[x=135pt,y=35pt,baseline={([yshift=-5pt]current bounding box.west)}]
  \node (12) at (1,1) {$\ast$};
  \node (21) at (0,0) {$\Map\!\big(X_\mathrm{dR}, B\PP GL_n\big)$};
  \node (22) at (1,0) {$\Map\!\big(X_\mathrm{dR}, B^2 \mu_m\big)$};
  \path[semithick,->]
    (12) edge[shorten <=-2pt] node[pos=0.4,scale=0.75,xshift=8pt] {$\theta'$} (22)
    (21) edge node[pos=0.5,scale=0.75,yshift=8pt] {$ob_{\mu_m}$} (22);
\end{tikzpicture}
\right\}
\]
and a subcategory of
\[
\lim \left\{
\begin{tikzpicture}[x=135pt,y=35pt,baseline={([yshift=-5pt]current bounding box.west)}]
  \node (12) at (1,1) {$\ast$};
  \node (21) at (0,0) {$\Map\!\big(X_\mathrm{Dol}, B\PP GL_n\big)$};
  \node (22) at (1,0) {$\Map\!\big(X_\mathrm{Dol}, B^2 \mu_m\big)$};
  \path[semithick,->]
    (12) edge[shorten <=-2pt] node[pos=0.4,scale=0.75,xshift=8pt] {$\widetilde\theta'$} (22)
    (21) edge node[pos=0.5,scale=0.75,yshift=8pt] {$ob_{\mu_m}$} (22);
\end{tikzpicture}
\right\}
\]
It is now a matter of
\begin{itemize}[topsep=-3pt,partopsep=0pt,parsep=2pt,itemsep=0pt]
\item studying the set of possible liftings, and
\item establishing the relationship between these limits and that those defining the categories $\Map_{B\GG_m} \!\big( {}_{\theta}(X_\mathrm{dR}), BGL_n \big) $ and $\Map_{B\GG_m} \!\big( {}_{\widetilde\theta}(X_\mathrm{Dol}), BGL_n \big)$.
\end{itemize}
Before doing this in \S\ref{section:Proof}, we investigate how to generalize \eqref{eq:MasterDiagramGL_n} to groups other than $GL_n$.

\subsection{A digression on algebraic groups}
\label{section:DisgressionOnAlgebraicGroups}

\subsubsection{}
\label{section:LinearAlgebraicGroups}

In this section we deal with linear (equivalently, affine) algebraic groups over $\CC$. Their theory, as worked out in, e.g., \cite{MR0302656}, regards them as sheaves on the big fppf site of $\operatorname{Spec}\CC$. Our first observation says that we can lift the whole theory to the {\' e}tale topology.

\begin{lemma}
\label{lemma:EtaleQuotientAlgebraicGroups}
Let $G$ be a linear algebraic group over $\CC$, and $N$ a closed normal subgroup. Then,
\begin{enumerate}[topsep=-3pt,partopsep=0pt,parsep=2pt,itemsep=0pt]
\item $N$ is a linear algebraic group,
\item the fppf quotient $G/N$ is representable by a linear algebraic group, and
\item the fppf quotient $G/N$ is also a quotient in the {\' e}tale topology.
\end{enumerate}
\end{lemma}

\begin{proof}
The first statement is obvious, since a closed subscheme of an affine scheme is affine. The second is \cite[III, \S 3, 5.6]{MR0302656}. To prove the third one, notice that the quotient map $G \to G/N$ is an fppf morphism \cite[III, \S 3, 2.5 a)]{MR0302656} with smooth fibers ---all of them are isomorphic to $N$, which is smooth by a theorem of Cartier's \cite[II, \S 6, 1.1 a)]{MR0302656}---, hence smooth; and smooth morphisms have sections {\' e}tale-locally.
\end{proof}

Should we want to stay on the algebraic ---by which of course we mean {\' e}tale--- side, the preceding lemma is all we need. If, however, we want to work in the analytic topology, we have to take one more step and use the analytification functor of \S\ref{section:Analytification}. But because the latter is exact and all of our constructions rely only on finite limits and colimits, the statements we use from their structure theory ---for which we refer to \cite{MR1102012} (see also \cite{milneAGS,milneRG})--- come through without any problem.

\subsubsection{}

Let $H$ be a \emph{connected} linear algebraic group over $\CC$, $A$ a closed subgroup contained in its center, and $K = H/A$. Because $A$ is abelian it decomposes as a product $A \cong F \times \GG_m^{\oplus r} \times \GG_a^{\oplus s}$ with $F$ finite abelian.

\begin{proposition}
\label{proposition:MasterDiagram}
There is a surjective homomorphism of linear algebraic groups $\kappa: H \to \GG_m^{\oplus r}$ such that the composition $A \hookrightarrow H \xrightarrow{\;\kappa\;} \GG_m^{\oplus r}$ is also surjective.
\end{proposition}

\begin{proof}
We can assume that $A$ coincides with the center of $H$; otherwise compose the homomorphism constructed below with a projection onto a torus of the appropriate rank in such a way that the induced map from the Lie algebra of $A$ to that of the torus becomes an isomorphism.

Write $H$ as an extension of its maximal reductive quotient $H_\mathrm{red}$ by its unipotent radical $R_u H$:
\[
1 \to R_u H \to H \to H_\mathrm{red} \to 1
\]
The structure theorem of reductive groups 
gives a decomposition of $H_\mathrm{red}$ as the almost-direct product of its radical $RH_\mathrm{red}$ and its derived subgroup $\mathcal{D}H_\mathrm{red}$:
\[
1 \to RH_\mathrm{red} \cap \mathcal{D}H_\mathrm{red} \to RH_\mathrm{red} \times \mathcal{D}H_\mathrm{red} \to H_\mathrm{red} \to 1
\]
Here $RH_\mathrm{red}$ is the maximal subtorus in $Z(H_\mathrm{red})$, and the intersection $RH_\mathrm{red} \cap \mathcal{D}H_\mathrm{red}$ is finite.

The canonical projection $RH_\mathrm{red} \times \mathcal{D}H_\mathrm{red} \longrightarrow RH_\mathrm{red}$ descends to a surjective homomorphism
\[
H_\mathrm{red} \longrightarrow \frac{RH_\mathrm{red}}{RH_\mathrm{red} \cap \mathcal{D}H_\mathrm{red}}
\]
that is obviously surjective when restricted to its center. Notice that the target of this map is again a torus, of the same rank as $RH_\mathrm{red}$. Composing with the projection $H \to H_\mathrm{red}$ provides the required homomorphism.
\end{proof}

\begin{remark}
\label{remark:ConnectednessOfH}
We insist on the connectedness assumption on $H$ only because we do not know whether we can extend this last proposition to the non-connected case. The requirement is superfluous whenever we can lift the homomorphism above from the connected component of the identity to the whole group. Since these fit in the exact sequence
\[
1 \to H^0 \to H \to \pi_0 H \to 1,
\]
all of our results hold, for example, for groups for which the latter is a split sequence.
\end{remark}

It is an instructive exercise to follow the steps of this last proof in the case $H = GL_n$, for it exactly reconstructs the determinant map. Remember that it was the surjectivity of this map ---and that of its restriction to the center of $GL_n$--- that allowed us to construct \eqref{eq:MasterDiagramGL_n}. Since $\kappa$ enjoys the same properties, we have another commutative diagram with exact rows and columns:
\begin{equation} \label{eq:MasterDiagramH}
\begin{gathered}
\begin{tikzpicture}[x=200pt, y=170pt]
  \node (12) at (0.2,1) {$1$};
  \node (13) at (0.4,1){$1$};
  \node (21) at (0,0.8) {$1$};
  \node (22) at (0.2,0.8) {$A'$};
  \node (23) at (0.4,0.8) {$H'$};
  \node (24) at (0.6,0.8) {$K$};
  \node (25) at (0.8,0.8) {$1$};
  \node (31) at (0,0.6) {$1$};
  \node (32) at (0.2,0.6) {$A$};
  \node (33) at (0.4,0.6) {$H$};
  \node (34) at (0.6,0.6) {$K$};
  \node (35) at (0.8,0.6) {$1$};
  \node (42) at (0.2,0.4) {$\GG_m^{\oplus r}$};
  \node (43) at (0.4,0.4) {$\GG_m^{\oplus r}$};
  \node (52) at (0.2,0.2) {$1$};
  \node (53) at (0.4,0.2) {$1$};
  \draw[double distance=1.5pt] (24) -- (34);
  \draw[double distance=1.5pt] (42) -- (43);
  \path[semithick,->]
    (12) edge (22)
    (13) edge (23)
    (21) edge (22)
    (22) edge (23)
    (22) edge (32)
    (23) edge (24)
    (23) edge (33)
    (24) edge (25)
    (31) edge (32)
    (32) edge (33)
    (32) edge node[pos=0.45,scale=0.8,xshift=7pt] {$\kappa$} (42)
    (33) edge (34)
    (33) edge node[pos=0.45,scale=0.8,xshift=7pt] {$\kappa$} (43)
    (34) edge (35)
    (42) edge (52)
    (43) edge (53);
\end{tikzpicture}
\end{gathered}
\end{equation}
Here we slightly abuse notation by denoting the restriction of $\kappa$ to $A$ by the same letter. There are two important remarks that we should make about the kernel $A'$ of the latter:
\begin{itemize}[topsep=-3pt,partopsep=0pt,parsep=2pt,itemsep=0pt]
\item By construction, $\kappa$ is trivial when restricted to the unipotent part of $A$, which is hence fully contained in $A'$.
\item On the other hand, the restriction of $\kappa$ to the non-unipotent part of $A$ induces an isomorphism at the level of Lie algebras, which in turn leads us to the crucial statement about $A'$: it contains no torus part. In other words, $A'$ satisfies the hypotheses of Proposition \ref{proposition:HodgeCorrespondenceGerbes}.
\end{itemize}

\subsection{Recapitulation: statement of the main theorem}
\label{section:StatementMainTheorems}

As promised in the introduction, the most general form of our main theorem involves basic \mbox{$H$-torsors} on $A$-gerbes over the de Rham and Dolbeault stacks of $X$, where $H$ is a linear algebraic group over $\CC$, and $A$ a closed subgroup of its center. Once again, we defer the statement of the stability conditions on the Dolbeault side until \S\ref{section:LiftingFlatHiggsAGerbes} (Definitions \ref{definition:SemistableBasicHPrimeTorsor} and \ref{definition:SemistableRectifiableBasicHTorsor}). Under the assumptions detailed below, the following statement holds.

\begin{theorem}
\label{theorem:MainTheorem}
Let $\theta \in \Map\!\big(X_\mathrm{dR}, B^2A\big)_0$ be a flat $A$-gerbe over $X$. Then there is a Higgs $A$-gerbe over X, $\widetilde\theta  \in \Map\!\big(X_\mathrm{Dol}, B^2A\big)_0$, for which there is an equivalence
\[
\Map_{BA} \!\big( {}_\theta(X_\mathrm{dR}), BH \big) \simeq \Map_{BA} \!\big( {}_{\widetilde\theta}(X_\mathrm{Dol}), BH \big)^\mathrm{ss}
\]
Conversely, given $\widetilde\theta \in \!\Map\big(X_\mathrm{Dol}, B^2A\big)_0$ we can find $\theta \in \!\Map\big(X_\mathrm{dR}, B^2A\big)_0$ such that the same conclusion holds.
\end{theorem}

The initial data is subject to the following two requirements:
\begin{itemize}[topsep=-3pt,partopsep=0pt,parsep=2pt,itemsep=0pt]
\item $H$ is connected, and
\item (in the algebraic category) the quotient $K = H/A$ is reductive.
\end{itemize}
The first condition ensures the existence of the map $\kappa$ in Proposition \ref{proposition:MasterDiagram} (see Remark \ref{remark:ConnectednessOfH}, though). The second ensures that the Hodge correspondence for $K$-torsors can be translated to the {\'e}tale case using Proposition \ref{proposition:AlgebraicHodgeCorrespondence}.

\section{The proof}
\label{section:Proof}

\subsection{Torusless gerbes and rectifiability}
\label{section:ToruslessGerbes}

In analogy with \eqref{eq:TorsionDiagramGL_n} we have the following diagrams of $\infty$-stacks, where again the rows and the column are fiber sequences:
\begin{equation} \label{eq:TorsionDiagramH}
\begin{gathered}
\begin{tikzpicture}[x=140pt, y=75pt]
  \node (11) at (0,1) {$BH'$};
  \node (12) at (0.5,1) {$BK$};
  \node (13) at (1,1) {$B^2A'$};
  \node (21) at (0,0.5) {$BH$};
  \node (22) at (0.5,0.5) {$BK$};
  \node (23) at (1,0.5) {$B^2A$};
  \node (33) at (1,0) {$B^2\GG_m^{\oplus r}$};
  \draw[double distance=1.5pt] (12) -- (22);
  \path[semithick,->]
    (11) edge (12)
    (11) edge (21)
    (12) edge node[pos=0.5,scale=0.8,yshift=10pt] {$ob_{A'}$} (13)
    (13) edge node[pos=0.45,scale=0.8,xshift=15pt] {$ob_A^{A'}$}(23)
    (21) edge (22)
    (22) edge node[pos=0.5,scale=0.8,yshift=10pt] {$ob_A$} (23)
    (23) edge node[pos=0.45,scale=0.8,xshift=8pt] {$\kappa$} (33);
\end{tikzpicture}
\end{gathered}
\end{equation}
It is now clear what takes on the role that torsion had in the vector bundle case.

\begin{definition}
We say that an $A$-gerbe over $\mathfrak{X}$ is \emph{$\kappa$-torsion} if its image under the map 
\[
\Map\!\big( \mathfrak{X}, B^2A \big) \xrightarrow{\;\;\;\kappa\;\;\;} \Map\!\big( \mathfrak{X}, B^2\GG_m^{\oplus r} \big),
\]
is a trivializable $\GG_m^{\oplus r}$-gerbe.
\end{definition}

\begin{lemma}[(cf. Lemma \ref{lemma:NoBasicGLnTorsorsForNonTorsionGerbes})] \label{lemma:NoBasicHTorsorsForNonKappaTorsionGerbes}
Let $\alpha \in \Map\!\big( \mathfrak{X}, B^2 A \big)_0$ be an $A$-gerbe over an $\infty$-stack $\mathfrak{X}$. Then, the category of basic $H$-torsors on ${}_\alpha \mathfrak{X}$ is empty unless ${}_\alpha \mathfrak{X}$ is $\kappa$-torsion.
\end{lemma}

\subsubsection{}
\label{section:ConstructionMultiplicationMap}

We now seek to give a different presentation of the category of basic $H$-torsors on a $\kappa$-torsion $A$-gerbe ---one that drops the explicit dependence on the category of $A$-gerbes in favor of that of $A'$-gerbes, for which the Hodge correspondence holds.

Consider the antidiagonal actions of $A'$ on $A' \times A$ and $H' \times A$. They give rise to an exact sequence of augmented simplicial objects in the category of linear algebraic groups over $\CC$:
\begin{equation} \label{eq:ExactSequenceAugmentedSimplicialObjects}
\begin{gathered}
\begin{tikzpicture}[x=450pt, y=45pt]
  \node (11) at (0,1) {$1$};
  \node (12) at (0.15,1) {$A' \times A \,\big/\!\!\big/ A'$};
  \node (13) at (0.35,1) {$H' \times A \,\big/\!\!\big/ A'$};
  \node (14) at (0.55,1) {$K \big/\!\!\big/ \ast$};
  \node (15) at (0.7,1) {$1$};
  \node (21) at (0,0) {$1$};
  \node (22) at (0.15,0) {$A$};
  \node (23) at (0.35,0) {$H$};
  \node (24) at (0.55,0) {$K$};
  \node (25) at (0.7,0) {$1$};
  \path[semithick,->]
    (11) edge (12)
    (12) edge node[scale=0.8,xshift=10pt] {$m$} (22)
    (12) edge (13)
    (13) edge (14)
    (13) edge node[scale=0.8,xshift=10pt] {$m$} (23)
    (14) edge (15)
    (14) edge[shorten <=2pt] (24)
    (21) edge (22)
    (22) edge (23)
    (23) edge (24)
    (24) edge (25);
\end{tikzpicture}
\end{gathered}
\end{equation}
The augmentations of each of these are in fact the quotient maps of the corresponding actions. Notice that, although the action of $A'$ described by the leftmost simplicial object seems to be twisted, it is actually isomorphic to the trivial $A'$-action on $A' \times A$.

The Puppe sequences of the rows in \eqref{eq:ExactSequenceAugmentedSimplicialObjects} yield the folllowing diagram of augmented simplicial objects in the appropriate $\infty$-topos of $\infty$-stacks:
\begin{equation} \label{eq:ObstructionDiagram}
\begin{gathered}
\begin{tikzpicture}[x=140pt, y=45pt]
  \node (11) at (0,1) {$BK \big/\!\!\big/ \ast$};
  \node (12) at (1,1) {$B^2A' \times B^2A \,\big/\!\!\big/ B^2A'$};
  \node (21) at (0,0) {$BK$};
  \node (22) at (1,0) {$B^2A$};
  \path[semithick,->]
    (11) edge node[scale=0.8,yshift=10pt] {$(ob_{A'},\ast) /\!/ \ast$} (12)
    (11) edge (21)
    (12) edge node[scale=0.8,xshift=10pt] {$m$} (22)
    (21) edge node[scale=0.8,yshift=10pt] {$ob_A$} (22);
\end{tikzpicture}
\end{gathered}
\end{equation}
Here the augmentation of the simplicial object on the left is trivially an effective epimorphism, while that of the simplicial object on the right is so because it is induced by the quotient of the trivial $A'$-torsor on $A$. The horizontal maps are composed of the universal obstruction $A'$- and $A$-gerbe for $K$-torsors, and the trivial maps that pick the natural points out of $B^2A'$ and $B^2A$.

Now, for any $\infty$-stack $\mathfrak{X}$, apply the $\infty$-functor $\Map\!\big( \mathfrak{X}, - \big)$ to \eqref{eq:ObstructionDiagram}. After choosing a $\kappa$-torsion $A$-gerbe on $\mathfrak{X}$, $\alpha \in \Map\!\big( \mathfrak{X}, B^2A \big)_0$, and a lift, $\alpha' \in \Map\!\big( \mathfrak{X}, B^2A' \big)_0$, to an $A'$-gerbe on $\mathfrak{X}$, we can extend the resulting diagram of augmented simplicial $\infty$-groupoids to the following:
\vspace{10pt}
\begin{equation}
\begin{gathered}
\begin{tikzpicture}[x=350pt, y=60pt]
  \node (11) at (0,1) {$\Map\!\big(\mathfrak{X}, BK\big) \big/\!\!\big/ \ast$};
  \node (12) at (0.55,1) {$\left.\left.\begin{matrix}\Map\!\big(\mathfrak{X}, B^2A'\big) \\[2pt] \times \Map\!\big(\mathfrak{X}, B^2A\big)\end{matrix} \,\right/\!\!\!\!\!\!\right/ \Map\!\big(\mathfrak{X}, B^2A'\big)$};
  \node (13) at (1,1) {$\ast \big/\!\!\big/ \ast$};
  \node (21) at (0,0) {$\Map\!\big(\mathfrak{X}, BK\big)$};
  \node (22) at (0.55,0) {$\Map\!\big(\mathfrak{X}, B^2A\big)$};
  \node (23) at (1,0) {$\ast$};
  \path[semithick,->]
    (11) edge node[scale=0.8,yshift=10pt] {$(ob_{A'},\ast) /\!/ \ast$} (12)
    (11) edge (21)
    (12) edge[shorten <=2pt] node[pos=0.5,scale=0.8,xshift=10pt] {$m$} (22)
    (13) edge node[scale=0.8,yshift=10pt] {$(\alpha',\ast) /\!/ \ast$} (12)
    (13) edge[shorten <=3pt] (23)
    (21) edge node[scale=0.8,yshift=10pt] {$ob_A$} (22)
    (23) edge node[scale=0.8,yshift=10pt] {$\alpha$} (22);
\end{tikzpicture}
\end{gathered}
\end{equation}
Once again, each vertical map in this diagram realizes the quotient of the groupoid object above it. Taking limits along the rows finally yields:
\vspace{10pt}
\begin{equation} \label{eq:MultiplicationMap}
\begin{gathered}
\begin{tikzpicture}[x=350pt, y=60pt]
  \node (1) at (0,1) {$\left.\left.\begin{matrix}\Map_{BA'}\!\big({}_{\alpha'}\mathfrak{X}, BH'\big) \\[2pt] \times \Map\!\big(\mathfrak{X}, BA\big)\end{matrix} \,\right/\!\!\!\!\!\!\right/ \Map\!\big(\mathfrak{X}, BA'\big)$};
  \node (2) at (0,0) {$\Map_{BA}\!\big({}_{\alpha}\mathfrak{X}, BH\big)$};
  \path[semithick,->]
    (1) edge node[pos=0.5,scale=0.8,xshift=10pt] {$m$} (2);
\end{tikzpicture}
\end{gathered}
\end{equation}
This last map is unfortunately \emph{not} an effective epimorphism\footnote{Effective epimorphisms constitute the \emph{left} part of an orthogonal factorization system in any $\infty$-topos, which is not necessarily preserved by $\infty$-limits \cite[\S 5.2.8]{MR2522659}}. It turns into one, however, if we restrict to certain connected components of the target. 

\subsubsection{}

Before exploring this last claim, let us describe what the objects and morphisms in \eqref{eq:MultiplicationMap} look like in concrete terms. From the realization \eqref{eq:BasicHTorsors} of the category of basic $H$-torsors on ${}_\alpha \mathfrak{X}$ as a limit, we see that its objects are given by pairs
\begin{equation} \label{eq:BasicHTorsor}
\Big( Q \to \mathfrak{X}, \; ob_A(Q \to \mathfrak{X}) \xrightarrow{\;\;\gamma\;\;} {}_\alpha\mathfrak{X} \Big)
\end{equation}
of a $K$-torsor on $\mathfrak{X}$ ---which we dub the \emph{underlying $K$-torsor} of the pair--- together with an equivalence between its obstruction $A$-gerbe and ${}_\alpha\mathfrak{X}$. A similar description could be made of basic $H'$-torsors on ${}_{\alpha'} \mathfrak{X}$, but there is a slightly different characterization of these that will later prove useful for our purposes.

Indeed, recall from \eqref{eq:TorsionDiagramH} that the following is a fiber sequence of $\infty$-groupoids:
\[
\Map\!\big( \mathfrak{X}, B^2 A' \big) \xrightarrow{\;\; ob_{A}^{A'}\;\;} \Map\!\big( \mathfrak{X}, B^2 A \big) \xrightarrow{\;\;\;\kappa\;\;\;} \Map\!\big( \mathfrak{X}, B^2 \GG_m^{\oplus r} \big).
\]
We can read this as saying that an $A'$-gerbe can be seen as a ($\kappa$-torsion) $A$-gerbe together with the choice of a trivialization of its image under $\kappa$. In particular, think of the obstruction $A'$-gerbe of a $K$-torsor $Q \to \mathfrak{X}$ as its obstruction $A$-gerbe together with a trivialization
\[
\kappa \big( ob_A(Q \to \mathfrak{X}) \big) \xrightarrow{\;\; ob_{A'}\;\;} \mathfrak{X} \times B\GG_m^{\oplus r}.
\]
Similarly, a lift $\alpha'$ of $\alpha$ can be expressed as the choice of a trivialization
\[
\kappa \big( {}_\alpha \mathfrak{X} \big) \xrightarrow{\;\;\alpha'\;\;} \mathfrak{X} \times B\GG_m^{\oplus r}.
\]
In this language, to give a basic $H'$-torsor on ${}_{\alpha'} \mathfrak{X}$ is the same as giving a triple
\begin{equation} \label{eq:BasicHPrimeTorsor}
\left( Q \to \mathfrak{X}, \; ob_A(Q \to \mathfrak{X}) \xrightarrow{\;\;\gamma\;\;} {}_\alpha\mathfrak{X},
\begin{gathered}
\begin{tikzpicture}[x=100pt, y=40pt]
  \node (1) at (0,1) {$\kappa \big( ob_A(Q \to \mathfrak{X}) \big)$};
  \node (2) at (1,1) {$\kappa \big( {}_\alpha \mathfrak{X} \big)$};
  \node (3) at (0.5,0) {$\mathfrak{X} \times B\GG_m^{\oplus r}$};
  \node[scale=1.25] (4) at (0.5,0.6) {$\Downarrow$};
  \node[scale=0.75] (5) at (0.58, 0.62) {$F$};
  \path[semithick,->]
    (1) edge node[scale=0.8, yshift=9pt] {$\kappa(\gamma)$} (2)
    (1) edge node[scale=0.8, pos=0.6, xshift=-16pt] {$ob_{A'}$} (3)
    (2) edge node[scale=0.8, pos=0.6, xshift=14pt] {$\alpha'$} (3);
\end{tikzpicture}
\end{gathered}
\right)
\end{equation}
The existence of a homotopy $F$ filling the last diagram is what says that $\gamma$ is part of an equivalence of $A'$-gerbes ---given by the choice of such a homotopy.

The first two pieces of data in \eqref{eq:BasicHPrimeTorsor} define a basic $H$-torsor on ${}_\alpha \mathfrak{X}$ \eqref{eq:BasicHTorsor}, and forgetting the homotopy $F$ is precisely the fortgetful functor
\begin{equation} \label{eq:ForgetfulMap}
\Map_{BA'}\!\big( {}_{\alpha'}\mathfrak{X}, BH' \big) \xrightarrow{\quad\pi\quad} \Map_{BA}\!\big( {}_\alpha \mathfrak{X}, BH \big)
\end{equation}
obtained by taking limits along the rows of the following diagram:
\[
\begin{tikzpicture}[x=250pt, y=40pt]
  \node (11) at (0,1) {$\Map\!\big(\mathfrak{X}, BK\big)$};
  \node (12) at (0.55,1) {$\Map\!\big(\mathfrak{X}, B^2A'\big)$};
  \node (13) at (1,1) {$\ast$};
  \node (21) at (0,0) {$\Map\!\big(\mathfrak{X}, BK\big)$};
  \node (22) at (0.55,0) {$\Map\!\big(\mathfrak{X}, B^2A\big)$};
  \node (23) at (1,0) {$\ast$};
  \draw[double distance=1.5pt] (11) -- (21); 
  \draw[double distance=1.5pt] (13) -- (23); 
  \path[semithick,->]
    (11) edge node[scale=0.8,yshift=8pt] {$ob_{A'}$} (12)
    (12) edge node[scale=0.8,xshift=16pt] {$ob_A^{A'}$} (22)
    (13) edge node[scale=0.8,yshift=8pt] {$\alpha'$} (12)
    (21) edge node[scale=0.8,yshift=8pt] {$ob_A$} (22)
    (23) edge node[scale=0.8,yshift=8pt] {$\alpha$} (22);
\end{tikzpicture}
\]
In the other direction, completing a pair $(Q, \gamma)$ to a triple $(Q, \gamma, F)$ is not a trivial task, but rather a strong condition on $(Q, \gamma)$ ---and one that depends on the choice of $\alpha'$. Indeed, for a fixed $\alpha'$, we have two trivializations of $\kappa \big( ob(Q \to \mathfrak{X}) \big)$: namely, $\alpha' \circ \kappa(\gamma)$ and $ob_{A'}$. But the category of trivializations of a trivializable $\GG_m^{\oplus r}$-gerbe is (noncanonically) equivalent to $\Map\!\big( \mathfrak{X}, B\GG_m^{\oplus r} \big)$, and we have no reason to expect the two to belong to the same connected component.

On the other hand, the category of equivalences between two (equivalent) $A$-gerbes is (again, noncanonically) equivalent to $\Map\!\big(\mathfrak{X}, BA\big)$, and so we have an action map
\begin{equation} \label{eq:ActionATorsorsOnBasicHTorsors}
\begin{gathered}
\begin{tikzpicture}[x=180pt, y=25pt]
  \node (11) at (0,1) {$\Map_{BA}\!\big( {}_\alpha \mathfrak{X}, BH \big) \times \Map\!\big( \mathfrak{X}, BA \big)$};
  \node (12) at (1,1) {$\Map_{BA}\!\big( {}_\alpha \mathfrak{X}, BH \big)$};
  \node (21) at (0,0) {$\big( (Q, \gamma), \mathcal{L} \big)$};
  \node (22) at (1,0) {$(Q, \mathcal{L}\gamma)$};
  \path[semithick,->]
    (11) edge[shorten <=4pt, shorten >=4pt] node[scale=0.8,yshift=8pt] {$\sigma$} (12);
  \path[semithick,|->]
    (21) edge[shorten <=4pt, shorten >=4pt] (22);
\end{tikzpicture}
\end{gathered}
\end{equation}
We can finally write the multiplication map $m$ in \eqref{eq:MultiplicationMap} as the appropriate composition of the forgetful functor \eqref{eq:ForgetfulMap} and this last map:
\[
\begin{gathered}
\begin{tikzpicture}[x=320pt, y=30pt]
  \node (10) at (-0.2,1) {$m:$};
  \node (11) at (0,1) {$\begin{matrix}\Map_{BA'}\!\big({}_{\alpha'}\mathfrak{X}, BH'\big) \\[2pt] \times \Map\!\big(\mathfrak{X}, BA\big)\end{matrix}$};
  \node (12) at (0.5,1) {$\begin{matrix}\Map_{BA}\!\big({}_\alpha \mathfrak{X}, BH\big) \\[2pt] \times \Map\!\big(\mathfrak{X}, BA\big)\end{matrix}$};
  \node (13) at (1,1) {$\Map_{BA}\!\big({}_\alpha \mathfrak{X}, BH\big)$};
  \node (21) at (0,0) {$\big( (Q, \gamma, F), \mathcal{L} \big)$};
  \node (22) at (0.5,0) {$\big( (Q, \gamma), \mathcal{L} \big)$};
  \node (23) at (1,0) {$(Q, \mathcal{L}\gamma)$};  
  \path[semithick,->]
    (11) edge node[scale=0.8,yshift=8pt] {$\pi\times\operatorname{id}$} (12)
    (12) edge node[scale=0.8,yshift=8pt] {$\sigma$} (13);
  \path[semithick,|->]
    (21) edge (22)
    (22) edge (23);    
\end{tikzpicture}
\end{gathered}
\]
\vspace{-10pt}

For the sake of completeness, let us also give an expression for the action of $\Map\!\big(\mathfrak{X}, BA'\big)$ in \eqref{eq:MultiplicationMap}. Seeing $A'$-torsors as pairs $(\mathcal{L}', \mu)$ consisting of an $A$-torsor and a trivialization of its image under $\kappa$ ---just as we did for $A'$-gerbes---, it is given by
\begin{equation} \label{eq:ActionOfAPrimeTorsors}
\begin{gathered}
\begin{tikzpicture}[x=200pt, y=25pt]
  \node (11) at (0,1) {$\begin{matrix}\Map_{BA'}\!\big({}_{\alpha'}\mathfrak{X}, BH'\big) \\[2pt] \times \Map\!\big(\mathfrak{X}, BA\big)\end{matrix} \times \Map\!\big( \mathfrak{X}, BA' \big)$};
  \node (12) at (1,1) {$\begin{matrix}\Map_{BA'}\!\big({}_{\alpha'}\mathfrak{X}, BH'\big) \\[2pt] \times \Map\!\big(\mathfrak{X}, BA\big)\end{matrix}$};
  \node (21) at (0,0) {$\big( (Q, \gamma, F), \mathcal{L}, (\mathcal{L}', \mu) \big)$};
  \node (22) at (1,0) {$\big( (Q, \mathcal{L'}\gamma, F\circ\mu), (\mathcal{L'})^{-1}\mathcal{L} \big)$};
  \path[semithick,->]
    (11) edge[shorten <=4pt, shorten >=4pt] node[scale=0.8,yshift=8pt] {$\mu$} (12);
  \path[semithick,|->]
    (21) edge[shorten <=4pt, shorten >=4pt] (22);
\end{tikzpicture}
\end{gathered}
\end{equation}

\subsubsection{}
\label{section:Rectifiability}

With descriptions in hand, we return to the claim at the end of \S\ref{section:ConstructionMultiplicationMap}.

\begin{definition} \label{definition:Rectifiability}
A basic $H$-torsor on ${}_\alpha \mathfrak{X}$ is said to be \emph{$\alpha'$-rectifiable} if it belongs to the image of $\Map_{BA'}\!\big({}_{\alpha'}\mathfrak{X}, BH'\big) \times \Map\!\big(\mathfrak{X}, BA\big)$ under the multiplication map $m$ in \eqref{eq:MultiplicationMap}.
\end{definition}

In other words, the multiplication map is an effective epimorphism onto the category of $\alpha'$-rectifiable basic $H$-torsors on ${}_\alpha \mathfrak{X}$, which, in turn, realizes the latter as the quotient
\[
\Map_{BA}\!\big({}_\alpha\mathfrak{X}, BH\big)^{\alpha'\text{-rect}} \simeq \frac{\Map_{BA'}\!\big({}_{\alpha'}\mathfrak{X}, BH'\big) \times \Map\!\big(\mathfrak{X}, BA\big)}{\Map\!\big(\mathfrak{X}, BA'\big)}
\]

\begin{proposition} \label{proposition:RectifiableBasicHTorsors}
A basic $H$-torsor on ${}_\alpha \mathfrak{X}$ is $\alpha'$-rectifiable if and only if the obstruction $A'$-gerbe of its underlying $K$-torsor is equivalent to ${}_{\alpha'} \mathfrak{X}$.
\end{proposition}

\begin{proof}
Necessity is clear, so we just need to prove sufficiency. Let $(Q, \gamma)$ be a basic $H$-torsor on ${}_\alpha \mathfrak{X}$, and suppose $ob_A'(Q \to \mathfrak{X})$ is equivalent to ${}_{\alpha'} \mathfrak{X}$. This amounts to the existence of a pair
\[
\left( ob_A(Q \to \mathfrak{X}) \xrightarrow{\;\;\phi\;\;} {}_\alpha\mathfrak{X},
\begin{gathered}
\begin{tikzpicture}[x=100pt, y=40pt]
  \node (1) at (0,1) {$\kappa \big( ob_A(Q \to \mathfrak{X}) \big)$};
  \node (2) at (1,1) {$\kappa \big( {}_\alpha \mathfrak{X} \big)$};
  \node (3) at (0.5,0) {$\mathfrak{X} \times B\GG_m^{\oplus r}$};
  \node[scale=1.25] (4) at (0.5,0.6) {$\Downarrow$};
  \node[scale=0.75] (5) at (0.58, 0.62) {$H$};
  \path[semithick,->]
    (1) edge node[scale=0.8, yshift=9pt] {$\kappa(\phi)$} (2)
    (1) edge node[scale=0.8, pos=0.6, xshift=-16pt] {$ob_{A'}$} (3)
    (2) edge node[scale=0.8, pos=0.6, xshift=14pt] {$\alpha'$} (3);
\end{tikzpicture}
\end{gathered}
\right)
\]
Since the category of equivalences between two (equivalent) $A$-gerbes is equivalent to $\Map\!\big(\mathfrak{X}, BA\big)$, we can find an $A$-torsor $\mathcal{L}$ such that $\phi \cong \mathcal{L}\gamma$. Hence $m\big( (Q, \phi, H), \mathcal{L} \big) = (Q, \mathcal{L}\phi) \cong (Q, \gamma)$.
\end{proof}

\begin{definition}
Let $[\alpha] \in \pi_0 \Map\!\big( \mathfrak{X}, B^2A \big)$. We define $L(\mathfrak{X})([\alpha])$ as the set of equivalence classes of liftings of the $A$-gerbe ${}_\alpha \mathfrak{X}$ to an $A'$-gerbe:
\[
L(\mathfrak{X})([\alpha]) := \Big\{ [\alpha'] \in \pi_0  \Map\!\big( \mathfrak{X}, B^2A' \big) \,\Big|\, ob_A^{A'}\!\!\left( [\alpha'] \right) = [\alpha] \Big\}
\]
\end{definition}

\begin{corollary} \label{Cor:UnionRectifiablePieces}
$\displaystyle \Map_{BA}\!\big({}_\alpha\mathfrak{X}, BH\big) \simeq \coprod_{[\alpha'] \in L(\mathfrak{X})([\alpha])} \Map_{BA}\!\big({}_\alpha\mathfrak{X}, BH\big)^{\alpha'\text{-rect}}$
\end{corollary}

\begin{proof}
By Proposition \ref{proposition:RectifiableBasicHTorsors}, every basic $H$-torsor $(Q, \gamma)$ is $ob_{A'}(Q \to \mathfrak{X})$-rectifiable.
\end{proof}

\noindent
Let us make two quick remarks about this sets of equivalence classes of liftings.
\begin{itemize}[topsep=-3pt,partopsep=0pt,parsep=2pt,itemsep=0pt]
\item The concept is only meaningful if ${}_\alpha \mathfrak{X}$ is $\kappa$-torsion, for otherwise $L(\mathfrak{X})([\alpha])$ is empty. Nevertheless, Corollary \ref{Cor:UnionRectifiablePieces} remains true in this case (cf. Lemma  \ref{lemma:NoBasicHTorsorsForNonKappaTorsionGerbes}).
\item Even if ${}_\alpha \mathfrak{X}$ is $\kappa$-torsion, it is not necessarily true that every element in $L(\mathfrak{X})([\alpha])$ is hit by an element of $\pi_0\Map\!\big( \mathfrak{X}, BK \big)$; that is, $\Map_{BA}\!\big({}_\alpha\mathfrak{X}, BH\big)^{\alpha'\text{-rect}}$ might be empty for some $[\alpha'] \in L(\mathfrak{X})([\alpha])$. We could therefore restrict the index set of the union in Corollary \ref{Cor:UnionRectifiablePieces} to the intersection of $L(\mathfrak{X})([\alpha])$ and the image of
\[
\pi_0\Map\!\big( \mathfrak{X}, BK \big) \xrightarrow{\;\; ob_{A'}\;\;} \pi_0\Map\!\big( \mathfrak{X}, B^2A' \big),
\]
without altering its validity.
\end{itemize}

\subsection{Lifting flat and Higgs \texorpdfstring{$A$}{A}-gerbes}
\label{section:LiftingFlatHiggsAGerbes}

\subsubsection{}

Choose a pair
\[
\theta' \in \Map\!\big( X_\mathrm{dR}, B^2A' \big)_0, \qquad \widetilde{\theta'} \in \Map\!\big( X_\mathrm{Dol}, B^2A' \big)_0
\]
consisting of a flat $A'$-gerbe on $X$ and a Higgs $A'$-gerbe on $X$ that are related to each other under the Hodge correspondence for gerbes (Proposition \ref{proposition:HodgeCorrespondenceGerbes}), and denote by
\[
\theta := ob^{A'}_A(\theta') \in \Map\!\big( X_\mathrm{dR}, B^2A \big)_0, \qquad \widetilde\theta := ob^{A'}_A(\widetilde{\theta'})  \in \Map\!\big( X_\mathrm{Dol}, B^2A \big)_0
\]
the induced ($\kappa$-torsion) $A$-gerbes.

\begin{definition}
\label{definition:SemistableBasicHPrimeTorsor}
A basic $H'$-torsor on ${}_{\widetilde{\theta'}}(X_\mathrm{Dol})$ is called \emph{semistable} if its underlying $K$-torsor on $X_\mathrm{Dol}$ is semistable and has zero first and second rational Chern classes. We denote by $\Map_{BA'}\!\big( {}_{\widetilde{\theta'}}(X_\mathrm{Dol}), BH' \big)^\mathrm{ss}$ the full subcategory of the category of basic $H'$-torsors on ${}_{\widetilde{\theta'}}(X_\mathrm{Dol})$ on the semistable objects.
\end{definition}

\begin{proposition} \label{proposition:CorrespondenceBasicHPrimeTorsors}
$\displaystyle \Map_{BA'}\!\big( {}_{\theta'}(X_\mathrm{dR}), BH' \big) \simeq \Map_{BA'}\!\big( {}_{\widetilde{\theta'}}(X_\mathrm{Dol}), BH' \big)^\mathrm{ss}$
\end{proposition}

\begin{proof}
This is just the obvious attempt that we described in \S\ref{section:ObviousAttempt} ---only now the Hodge correspondence does hold for $A'$-gerbes.
\end{proof}

\begin{definition}
\label{definition:SemistableRectifiableBasicHTorsor}
A $\widetilde{\theta'}$-rectifiable $H$-torsor on ${}_{\widetilde\theta} (X_\mathrm{Dol})$ is called \emph{semistable} if it is the image under the multiplication map $m$ in \eqref{eq:MultiplicationMap} of a semistable basic $H'$-torsor on ${}_{\widetilde{\theta'}}(X_\mathrm{Dol})$ and an $A$-torsor on $X_\mathrm{Dol}$ with zero first Chern class. We denote the category of these objects by
\[
\Map_{BA}\!\big( {}_{\widetilde\theta} (X_\mathrm{Dol}), BH \big)^{\widetilde{\theta'}\text{-rect}, \mathrm{ss}}
\]
\end{definition}

Observe that, owing to the discussion in \S\ref{section:SemistabilityForAbelianTorsors}, we do not need to explicitly require semistability for the $A$-torsor on $X_\mathrm{Dol}$ ---vanishing of the first Chern class is enough.

Since the multiplication map $m$ in \eqref{eq:MultiplicationMap} is a quotient map, there is the question of whether this definition of semistability depends on the choice of inverse image. Any two objects of $\Map_{BA'}\!\big( {}_{\widetilde{\theta'}}(X_\mathrm{Dol}), BH' \big) \times \Map\!\big( X_\mathrm{Dol}, BA \big)$ giving rise to the same basic $H$-torsor on ${}_{\widetilde\theta} (X_\mathrm{Dol})$ differ by the action of an $A'$-torsor on $X_\mathrm{Dol}$ as in \eqref{eq:ActionOfAPrimeTorsors}. However, the latter are always of degree zero because $A'$ contains no algebraic torus, and they remain of degree zero after taking their image under $ob^{A'}_A$. Hence, if one of the objects consists of a semistable basic $H'$-torsor on ${}_{\widetilde{\theta'}}(X_\mathrm{Dol})$ and an $A$-torsor on $X_\mathrm{Dol}$ with zero first Chern class, so does the other. In other words, the action of $\Map\!\big( X_\mathrm{Dol}, BA' \big)$ preserves the subcategory 
\[
\Map_{BA'}\!\big( {}_{\widetilde{\theta'}}(X_\mathrm{Dol}), BH' \big)^\mathrm{ss} \times \Map\!\big( X_\mathrm{Dol}, BA \big)^0 \subseteq \Map_{BA'}\!\big( {}_{\widetilde{\theta'}}(X_\mathrm{Dol}), BH' \big) \times \Map\!\big( X_\mathrm{Dol}, BA \big).
\]

\begin{corollary} \label{cor:HodgeCorrespondenceRectifiableTorsors}
$\Map_{BA}\!\big( {}_{\theta} (X_\mathrm{dR}), BH \big)^{\theta'\text{-rect}} \simeq \Map_{BA}\!\big( {}_{\widetilde\theta} (X_\mathrm{Dol}), BH \big)^{\widetilde{\theta'}\text{-rect}, \mathrm{ss}}$
\end{corollary}
\vspace{-5pt}
\begin{proof}
Once again, the obvious approach works: the terms in the definitions of both of the categories on both sides of this equivalence,
\[
\Map_{BA}\!\big( {}_{\theta} (X_\mathrm{dR}), BH \big)^{\theta'\text{-rect}} \simeq \frac{\Map_{BA'}\!\big( {}_{\theta'}(X_\mathrm{dR}), BH' \big) \times \Map\!\big( X_\mathrm{dR}, BA \big)}{\Map\!\big( X_\mathrm{dR}, BA' \big)}
\]
and
\[
\Map_{BA}\!\big( {}_{\widetilde\theta} (X_\mathrm{Dol}), BH \big)^{\widetilde{\theta'}\text{-rect}, \mathrm{ss}} \simeq \frac{\Map_{BA'}\!\big( {}_{\widetilde{\theta'}}(X_\mathrm{Dol}), BH' \big)^\mathrm{ss} \times \Map\!\big( X_\mathrm{Dol}, BA \big)^0}{\Map\!\big( X_\mathrm{Dol}, BA' \big)},
\]
exactly correspond to each other under Theorem \ref{th:NonabelianHodgeTheoremGTorsors}.
\end{proof}

\subsubsection{}

In the last paragraph we proved something that puts us very close to the statement of Theorem \ref{theorem:MainTheorem}. Indeed, for any choice of a $\kappa$-torsion flat $A$-gerbe on $X$, 
$\theta \in \Map\!\big( X_\mathrm{dR}, B^2A \big)_0$, Corollary \ref{Cor:UnionRectifiablePieces} gives us the decomposition
\[
\Map_{BA}\!\big({}_\theta (X_\mathrm{dR}), BH\big) \simeq \coprod_{[\theta'] \in L(X_\mathrm{dR})([\theta])} \Map_{BA}\!\big( {}_{\theta} (X_\mathrm{dR}), BH \big)^{\theta'\text{-rect}}
\]
By Corollary \ref{cor:HodgeCorrespondenceRectifiableTorsors}, each one of the pieces on the right hand side of the last equation is equivalent to the category of semistable $\widetilde{\theta'}$-rectifiable basic $H$-torsors on ${}_{ob^{A'}_A (\widetilde{\theta'})} (X_\mathrm{Dol})$, where $\widetilde{\theta'}$ is related to $\theta'$ through the Hodge correspondence for gerbes. Our first remark is that all of the $[ob^{A'}_A (\widetilde{\theta'})]$ coincide. 

\begin{lemma} \label{lemma:CorrespondenceComponents}
Let $[\theta] \in \pi_0\Map\!\big( X_\mathrm{dR}, B^2A \big)$ be an equivalence class of $\kappa$-torsion flat $A$-gerbes on $X$. There is a unique equivalence class of $\kappa$-torsion Higgs $A$-gerbes on $X$, $[\widetilde\theta] \in \pi_0\Map\!\big( X_\mathrm{Dol}, B^2A \big)$, for which there is a injective function
\[
L(X_\mathrm{dR})([\theta]) \hookrightarrow L(X_\mathrm{Dol})([\widetilde\theta])
\]
\vspace{-20pt}
\end{lemma}

\begin{proof}
Consider the following diagram, constructed out of the Puppe sequence of the exact sequence $0 \to A' \to A \to \GG_m^{\oplus r} \to 0$ of abelian linear algebraic groups over $\CC$ coming from \eqref{eq:MasterDiagramH}.
\begin{equation} \label{eq:Liftings}
\begin{gathered}
\begin{tikzpicture}[x=300pt, y=120pt]
  \node (11) at (0,1) {$\pi_0\Map\!\big( X_\mathrm{dR}, BA \big)$};
  \node (12) at (0.5,1) {$\pi_0\Map\!\big( X_\mathrm{Dol}, BA \big)^0$};
  \node (13) at (1,1) {$\pi_0\Map\!\big( X_\mathrm{Dol}, BA \big)$};
  \node (21) at (0,0.666) {$\pi_0\Map\!\big( X_\mathrm{dR}, B\GG_m^{\oplus r} \big)$};
  \node (22) at (0.5,0.666) {$\pi_0\Map\!\big( X_\mathrm{Dol}, B\GG_m^{\oplus r} \big)^0$};
  \node (23) at (1,0.666) {$\pi_0\Map\!\big( X_\mathrm{Dol}, B\GG_m^{\oplus r} \big)$};
  \node (31) at (0,0.333) {$\pi_0\Map\!\big( X_\mathrm{dR}, B^2A' \big)$};
  \node (32) at (0.5,0.333) {$\pi_0\Map\!\big( X_\mathrm{Dol}, B^2A' \big)$};
  \node (33) at (1,0.333) {$\pi_0\Map\!\big( X_\mathrm{Dol}, B^2A' \big)$};
  \node (41) at (0,0) {$\pi_0\Map\!\big( X_\mathrm{dR}, B^2A \big)$};
  \node (42) at (0.5,0) {$\pi_0\Map\!\big( X_\mathrm{Dol}, B^2A \big)$};
  \node (43) at (1,0) {$\pi_0\Map\!\big( X_\mathrm{Dol}, B^2A \big)$};
  \node[scale=1.5] (cartesian) at (0.54,0.86) {$\ulcorner$};
  \draw[double distance=1.5pt] (32) -- (33); 
  \draw[double distance=1.5pt] (42) -- (43);
  \path[semithick,->]
    (11) edge node[scale=0.8,yshift=8pt] {$\cong$} (12)
    (11) edge (21)
    (12) edge (22)
    (13) edge (23)
    (21) edge node[scale=0.8,yshift=8pt] {$\cong$} (22)
    (21) edge (31)
    (22) edge (32)
    (23) edge (33)
    (21) edge (22)
    (31) edge node[scale=0.8,yshift=8pt] {$\cong$} (32)
    (31) edge node[scale=0.8,xshift=14pt] {$ob^{A'}_A$} (41)
    (32) edge node[scale=0.8,xshift=14pt] {$ob^{A'}_A$} (42)
    (33) edge node[scale=0.8,xshift=14pt] {$ob^{A'}_A$} (43);
  \path[semithick,{Hooks[right]}->]
    (12) edge (13)
    (22) edge (23);
\end{tikzpicture}
\end{gathered}
\end{equation}
The first and last columns are, by construction, exact sequences of abelian groups; although the middle one is not exact, the composition of any two maps in it is zero. The maps from the first column to the second are provided by the appropriate Hodge correspondences.

For $[\theta] \in \pi_0\Map\!\big( X_\mathrm{dR}, B^2A \big)$, look at all its possible liftings through $ob^{A'}_A$ ---that is, at $L(X_\mathrm{dR})([\theta])$. Since any two elements in the latter set differ by an element in $\pi_0\Map\!\big( X_\mathrm{dR}, B\GG_m^{\oplus r} \big)$, their images under the Hodge correspondence differ by an element of $\pi_0\Map\!\big( X_\mathrm{Dol}, B\GG_m^{\oplus r} \big)^0$. These images then map to a well-defined element $[\widetilde\theta] \in \pi_0\Map\!\big( X_\mathrm{Dol}, B^2A \big)$, thus defining a function $L(X_\mathrm{dR})([\theta]) \to L(X_\mathrm{Dol})([\widetilde\theta])$. The injectivity of the latter is clear.
\end{proof}

This last map is in general not surjective. Indeed, the upper right square in \eqref{eq:Liftings} being cartesian, we have
\begin{multline*}
L(X_\mathrm{dR})([\theta]) \cong \frac{\pi_0\Map\!\big( X_\mathrm{dR}, B\GG_m^{\oplus r} \big)}{\pi_0\Map\!\big( X_\mathrm{dR}, BA \big)} \cong \frac{\pi_0\Map\!\big( X_\mathrm{Dol}, B\GG_m^{\oplus r} \big)^0}{\pi_0\Map\!\big( X_\mathrm{Dol}, BA \big)^0} \\ \hookrightarrow  \frac{\pi_0\Map\!\big( X_\mathrm{Dol}, B\GG_m^{\oplus r} \big)}{\pi_0\Map\!\big( X_\mathrm{Dol}, BA \big)} \cong L(X_\mathrm{Dol})([\widetilde\theta])
\end{multline*}
However, it becomes an isomorphism when restricted to the images under $ob_{A'}$ of the category of $K$-torsors on $X_\mathrm{dR}$, on one side, and that of the category of semistable $K$-torsors on $X_\mathrm{Dol}$ with zero first and second rational Chern classes, on the other:
\begin{multline*}
\operatorname{im} \left( \pi_0\Map\!\big( X_\mathrm{dR}, BK \big) \xrightarrow{\;\; ob_{A'}\;\;} \pi_0\Map\!\big( X_\mathrm{dR}, B^2A' \big) \right) \cap L(X_\mathrm{dR})([\theta]) \\
\cong \operatorname{im} \left( \pi_0\Map\!\big( X_\mathrm{Dol}, BK \big)^{\mathrm{ss}, 0} \xrightarrow{\;\; ob_{A'}\;\;} \pi_0\Map\!\big( X_\mathrm{Dol}, B^2A' \big) \right) \cap L(X_\mathrm{Dol})([\widetilde\theta])
\end{multline*}
This fact follows easily from the commutative diagram
\[
\begin{tikzpicture}[x=140pt, y=45pt]
  \node (11) at (0,1) {$\pi_0\Map\!\big( X_\mathrm{dR}, BK \big)$};
  \node (12) at (1,1) {$\pi_0\Map\!\big( X_\mathrm{Dol}, BK \big)^{\mathrm{ss}, 0}$};
  \node (21) at (0,0) {$\pi_0\Map\!\big( X_\mathrm{dR}, B^2A' \big)$};
  \node (22) at (1,0) {$\pi_0\Map\!\big( X_\mathrm{Dol}, B^2A' \big)$};
  \path[semithick,->]
    (11) edge node[scale=0.8,yshift=8pt] {$\cong$} (12)
    (11) edge node[scale=0.8,xshift=14pt] {$ob_{A'}$}(21)
    (12) edge node[scale=0.8,xshift=14pt] {$ob_{A'}$} (22)
    (21) edge node[scale=0.8,yshift=8pt] {$\cong$} (22);
\end{tikzpicture}
\]
mediated by the nonabelian Hodge correspondence, and the Hodge correspondence for gerbes. Because of the last remark of \S\ref{section:Rectifiability}, this observation finishes the proof of Theorem \ref{theorem:MainTheorem}. In particular, it allows us to find, given an equivalence class of $\kappa$-torsion Higgs $A$-gerbes on $X$, an equivalence class of $\kappa$-torsion flat $A$-gerbes for which the conclusion of Lemma \ref{lemma:CorrespondenceComponents} holds.

\nocite{arXiv1207.0249,MR0354653,Escourido08}
\bibliographystyle{amsalpha_manualtags}
\bibliography{TwistedNAHT}

\end{document}